\pdfoutput=1 
\documentclass{amsart}
\usepackage{amssymb,amsfonts}
\usepackage{amsthm}
\usepackage{amsmath} 
\usepackage{amsthm} 
\usepackage{mathtools}
\usepackage[all,arc]{xy}
\usepackage{enumerate}
\usepackage{mathrsfs}
\usepackage{hyperref}

\hypersetup{
pdftitle={On the link between finite difference and derivative of polynomials},
pdfsubject={Mathematics, Differential Calculus, Derivatives, Differential equations},
pdfauthor={Poramate Nakkirt},
pdfkeywords={The Watt-Strogatz random graph, small-world random graph, adjacency matrix, random matrix, eigenvalue distribution, method of moments}
}
\usepackage[utf8]{inputenc}
\DeclareUnicodeCharacter{2217}{-}

\newtheorem{thm}{Theorem}[section]

\newtheorem{prop}[thm]{Proposition}
\newtheorem{lem}[thm]{Lemma}

\theoremstyle{definition}
\newtheorem{defn}[thm]{Definition}

\newtheorem{notn}[thm]{Notation}

\theoremstyle{remark}
\newtheorem{rem}[thm]{Remark}

\makeatletter
\let\c@equation\c@thm
\makeatother
\numberwithin{equation}{section}
\title{The Eigenvalue Distribution of the Watts-Strogatz Random Graph}

\author{Poramate Nakkirt}
\begin{document}
\footnote{This paper is based on the author's undergraduate honors thesis at the University of Colorado Boulder, available on \url{https://scholar.colorado.edu/concern/undergraduate_honors_theses/cz30pt78x}. Unlike this paper, the thesis version \cite{7} is only intended for consideration for Latin honors for the bachelor's degree in mathematics, not intended for publication.}
\begin{abstract}
This paper studies the eigenvalue distribution of the Watts-Strogatz random graph, which is known as the ``small-world" random graph. The construction of the small-world random graph starts with a regular ring lattice of $n$ vertices; each has exactly $k$ neighbors with equally $\frac{k}{2}$ edges on each side. With probability $p$, each downside neighbor of a particular vertex will rewire independently to a random vertex on the graph without allowing for self-loops or duplication. The rewiring process starts at the first adjacent neighbor of vertex $1$ and continues in an orderly fashion to the farthest downside neighbor of vertex $n$. Each edge must be considered once. We focus on the eigenvalues of the adjacency matrix $A_n$, used to represent the small-world random graph. The moments generally decide its distribution. 
We compute the first moment, second moment, and prove the limiting third moment as $n \to \infty$ of the eigenvalue distribution. 
\\\\
\end{abstract}

\maketitle
\tableofcontents
\section{Introduction}\label{d1}
The Watts-Strogatz random graph is usually called the ``small-world" random graph.  This random graph was discovered by Watts and Strogatz in 1998 who aimed to study the behavior of a random graph that interpolates between a regular graph and a (highly-disordered) random graph \cite{9}. In \cite{10}, Watts and Strogatz constructed a small-world random graph by rewiring some edges in a regular ring lattice with $n$ vertices and degree $k$. However, even though Watts and Strogatz introduced a new construction of the random graph, their graph still preserves two properties: high clustering (like a regular graph) and low average path length or the average number of separation between two vertices (like a highly-disordered graph) \cite{4}. The following is the construction of a small-world random graph $\mathbf{G}$
\cite{1}\cite{3}\cite{7}\cite{8}\cite{9}\cite{10}\cite{12}. 
\newline
\noindent \textbf{Define:}
\begin{enumerate}
    \item $\mathbf{N}(i)$ is a set of all vertices $v$ such that the edge $\{i,v\}$ is in the graph.
    \item The vertex $i \pm d$ for any $d \in \mathbb{N}$ to represent the vertex $i \pm d \mathrm{\ (mod\ n)}$.
\end{enumerate}
\textbf{Required:}
\begin{enumerate}
    \item The parameters $n \in \mathbb{}N$, $k \in 2\mathbb{N}$, and $p \in [0,1]$.
    \item The undirected regular ring lattice on the vertex set $\{1,2,...,n\}$ with the degree $k \in 2\mathbb{N}$, where for each vertex half of the edges $(\frac{k}{2} \in \mathbb{N})$ are on the upside and half of the edges $(\frac{k}{2} \in \mathbb{N})$ are on the downside.
\end{enumerate}
\textbf{Algorithm:}
\begin{itemize}
    \item Consider vertex $i$ and the edges $\{i,j\}$ for $j = i+1,i+2,...,i+\frac{k}{2}$
    \begin{itemize}
        \item With probability $1-p$, we keep the edge $\{i,j\}$.
        \item Otherwise,
        \begin{itemize}
            \item The vertex $j'$ is chosen uniformly at random from \\ $\{1,2,...,n\} \backslash (\{i-\frac{k}{2},...,i-1,i,i+1,...,i+\frac{k}{2} \} \cup \mathbf{N}(i))$, to guarantee that the edge $\{i,j'\}$ does not make a self-loop or duplication. 
            \item Replace the edge $\{i,j\}$ by $\{i,j'\}$.
        \end{itemize}
    \end{itemize}
    \item Repeat this algorithm until all vertices $i = 1,2,...,n$ have been considered once.
    \item \textbf{Output:} \textbf{G}
\end{itemize}

\begin{defn}
Given three parameters $n \in \mathbb{N}$ is the total number of vertices, $k \in 2\mathbb{N}$ is the number of each vertex's neighbor (degree), and $p \in [0,1]$ is the rewiring probability. Let $SW(n,k,p)$ represents a small-world random graph that is created by above \textbf{Algorithm}. 
\end{defn}
\begin{rem}
 In this random graph, we assume $n \gg  k$ \cite{4}\cite{10}. 
\end{rem}
\noindent We can see the examples of $SW(n,k,p)$ random graph in \cite{7}. 
\section{The Eigenvalue Distribution}
When we create a small-world random graph, it is important to know how to study the eigenvalue distribution of the random graph. We begin with representing a small-world graph by the adjacency matrix and use the method of moments to primarily study the behavior of the eigenvalue distribution and properties of the small-world random graph. 
\begin{defn}
Let $\{1,2,...,n\}$ be a set of vertices of the graph. The adjacency matrix $A_n$ is the square $n \times n$ matrix such that its elements are $1$ or $0$ based on if any two vertices are adjacent or not.\\
For $i,j \in \{1,2,...,n\}$,
\begin{equation}
    A_{ij} = \left\{\begin{array}{lr}
        1, & \mathrm{the \ edge \ }\{i,j\} \mathrm{\ is \ in\ graph\ } \\
        0, & \mathrm{Otherwise}
        \end{array}\right\}
\end{equation}
\end{defn}
The notation $\sim$ signifies the adjacency matrix being used to represent the random graph. For instance, $M \sim SW(n,k,p)$ means the adjacency matrix $M$ represents the small-world random graph with given parameters $n,k,p$. \\
\indent For the small-world random graph, all edges on the graph are undirected. For any adjacency matrix $A_n \sim SW(n,k,p)$, the entries $A_{ij} = A_{ji}$ since an edge $\{i,j\}$ is the same as $\{j,i\}$. 

\begin{prop}
    For the small-world random graph, let $A_n \sim SW(n,k,p)$ for $n,k,p$ are constants, then $A_n$ is symmetric. 
\end{prop}

\indent Another observation is the diagonal entries $A_{ii} = 0$ for all $i \in \{1,2,...,n\}$ since the \textbf{Algorithm} does not allow a self-loop. 

\begin{prop}
    Let $A_n \sim SW(n,k,p)$. Then the diagonal entries of $A_n$ are all zero. 
\end{prop}

\indent Based on \textbf{Algorithm}, we know that each vertex $i \in \{1,2,...,n\}$ can connect to exactly k other neighbors for a regular ring lattice. By Definition $2.1$, the sum of all entries of the adjacency matrix $A_n$ is $nk$. Since $A_n$ is symmetric, each entries $A_{ij}$ will be counted twice with $A_{ji}$. Thus, the total number of edges in the graph is half of the sum of all entries of $A_n$ $(\frac{nk}{2})$. After the rewiring process is done, the graph still has the same total number of edges because for every removal of an edge, an additional edge must be connected. 

\begin{prop}
    Let $A_n \sim SW(n,k,p)$, then the sum of all entries in $A_n$ is $nk$ and the number of all edges is $\frac{nk}{2}$.
\end{prop}

\section{The Method of Moments}
This section discusses the main method that we bring to study the eigenvalue distribution of the small-world random graph. In \cite{11}, Tao states the importance of the method of moments to prove the behavior and characteristics of the eigenvalue distribution. 
He also provides a formula to compute a general $l^{th}$ moment for $l \in \mathbb{N}$ as a starting point to study eigenvalues. Since we work on the case when the matrix is symmetric (see Proposition 2.3), so all eigenvalues are real numbers. \\
\noindent The notation Tr(M) means the trace of the square matrix $M$. \\
\indent Let $A_n$ be the adjacency matrix of the random graph. Let $\lambda_1,\lambda_2,...,\lambda_n \in \mathbb{R}$ be all eigenvalues of $A_n$. By the matrix identity in linear algebra, for any $l \in \mathbb{N}$, we have the equation
\begin{equation}
\mathrm{Tr}(A_n^l) = \sum_{i=1}^n \lambda_i^l.
\end{equation}
Since we need to study $\sum_{i=1}^n \lambda_i^l$ scaling by $n$, hence it follows that
\begin{equation}
\frac{1}{n}\mathrm{Tr}(A_n^l) = \frac{1}{n}\sum_{i=1}^n \lambda_i^l.
\end{equation}
We take the expectation to the above equation, and we have
\begin{equation}
\mathbb{E}[\frac{1}{n}\mathrm{Tr}(A_n^l)] = \mathbb{E}[\frac{1}{n}\sum_{i=1}^n \lambda_i^l].
\end{equation}
By Tao's equation (2.70) in \cite{11}, we have
\begin{equation}
\mathbb{E}[\frac{1}{n}\mathrm{Tr}(A_n^l)] = \frac{1}{n}\mathbb{E}[\mathrm{Tr}(A_n^l)] = \frac{1}{n}\sum_{1 \le i_1,...,i_l \le n} \mathbb{E}[A_{i_1i_2}A_{i_2i_3}A_{i_3i_4}...A_{i_li_1}],
\end{equation}
which is the sum over the expectation of the cycles of entries multiplication of length $l$, and scaling by $n$.
The following formulas are the first three moments of the eigenvalue distribution of the adjacency matrix  $A_n$ that will be used later in the paper. \\
The first moment of the adjacency matrix $A_n$ is
\begin{equation}
    \mathbb{E}[\frac{1}{n}\mathrm{Tr}(A_n)] = \frac{1}{n}\sum_{i=1}^n \mathbb{E}[A_{ii}].
\end{equation}
The second moment of the adjacency matrix $A_n$ is
\begin{equation}
    \mathbb{E}[\frac{1}{n}\mathrm{Tr}(A_n^2)] = \frac{1}{n}\sum_{1 \le i_1,i_2 \le n} \mathbb{E}[A_{i_1i_2}A_{i_2i_1}].
\end{equation}
The Third moment of the adjacency matrix $A_n$ is
\begin{equation}
    \mathbb{E}[\frac{1}{n}\mathrm{Tr}(A_n^3)] = \frac{1}{n}\sum_{1 \le i_1,i_2,i_3 \le n} \mathbb{E}[A_{i_1i_2}A_{i_2i_3}A_{i_3i_1}].
\end{equation}

\indent Let $A_n \sim SW(n,k,p)$ with fixed values  $n \in \mathbb{N},k \in 2\mathbb{N}$, and $p \in [0,1]$. After we convert the adjacency matrix $A_n$ from the small-world random graph with given parameters $n,k,p$, we compute all real eigenvalues. Then we use all eigenvalues to plot the histogram of the eigenvalue density. Finally, we observe and investigate the behaviors and characteristics of a given distribution when we vary all three parameters $n,k$, and $p$. Each particular value of input $(n,k,p)$ gives different shape of the distribution. Results with more detailed pictures are available on \cite{7}. Those motivate us the following theorems.
\begin{thm}
Given $n \in \mathbb{N}$ is an arbitrary and $\frac{k}{2} \in \mathbb{N}$ and $p \in [0,1]$ are fixed. Let $A_n \sim SW(n,k,p)$ be the adjacency matrix that represents the small-world random graph. Then the first and second moments of the eigenvalue distribution are  
\begin{equation}
    \frac{1}{n} \mathrm{Tr}(A_n) = 0
\end{equation}
\begin{equation}
    \frac{1}{n} \mathrm{Tr}(A_n^2) = k.
\end{equation}
\end{thm}

\begin{thm}
Given $n \in \mathbb{N}$ is an arbitrary and $\frac{k}{2} \in \mathbb{N}$ and $p \in [0,1]$ are fixed. Let $A_n \sim SW(n,k,p)$ be the adjacency matrix that represents the small-world random graph. Then, the limiting third moment of the eigenvalue distribution is 
\begin{equation}
    \lim_{n \to \infty} \mathbb{E}[\frac{1}{n} \mathrm{Tr}(A_n^3)] = \frac{3k(k-2)(1-p)^3}{4}.
\end{equation}
\end{thm}

\section{The First and Second Moments}
Let $A_n \sim SW(n,k,p)$ be the adjacency matrix represents the Watts-Strogatz random graph.
We begin with the proof of the first moment (3.9) in Theorem 3.8.
\begin{proof} By (3.5), the algebraic formula of the first moment, which is the trace of $A_n$ scaled by $n$, is equivalent to the sum of diagonal entries of the matrix $A_n$ scaled by $n$. Hence,
\begin{eqnarray}
\frac{1}{n} \mathrm{Tr}(A_n) & = & \frac{1}{n} \sum_{1 \le i \le n} A_{ii}  \\ 
& = & \frac{1}{n}(0) \\ 
&=& 0
\end{eqnarray}
In (4.2), it holds by Proposition 2.4. Therefore, it proves that $\frac{1}{n} \mathrm{Tr}(A_n) = 0$. 
\end{proof}
Next, we prove the second moment (3.10) in Theorem 3.8.
\begin{proof} We use $(3.6)$ formula to compute the second moment. That is,
\begin{eqnarray}
    \frac{1}{n} \mathrm{Tr}(A_n^2) = \frac{1}{n} \sum_{1 \le i_1,i_2 \le n} A_{i_1i_2}A_{i_2i_1} \\
    = \frac{1}{n} \sum_{1 \le i_1,i_2 \le n} A_{i_1i_2}A_{i_1i_2},
\end{eqnarray}
where the last equality holds by Proposition 2.3. Then,
\begin{eqnarray}
\frac{1}{n} \mathrm{Tr}(A_n^2)
&=& \frac{1}{n} \sum_{1 \le i_1,i_2 \le n} A_{i_1i_2}^2 \\
&=& \frac{1}{n}\sum_{1 \le i_1,i_2 \le n} A_{i_1i_2}  
\end{eqnarray}
\noindent In (4.7), the result holds since the entries of $A_n$ are either 1 or 0. In addition, we can observe in $(4.4)$ that $A_{i_1i_2} = A_{i_2i_1}$ since 
the edges $\{i_1,i_2\}$ and $\{i_2,i_1\}$ are the same.
Next, for $1 \le i_1,i_2 \le n$, we have $$\frac{1}{n} \mathrm{Tr}(A_n^2) = \frac{1}{n}\sum_{i_1} \sum_{i_2} A_{i_1i_2}
= \frac{1}{n}\sum_{i_1}(A_{i_1 1} +...+A_{i_1 n}).$$ 
By Propositions 2.3 and 2.5, the sum of all entries of $A_n$ is $nk$. Therefore,
$$\frac{1}{n} \mathrm{Tr}(A_n^2) = \frac{1}{n}(2)(\frac{nk}{2}) = \frac{1}{n}(nk) \\ = k.$$
\end{proof}

\section{The Third Moment Formula}
This section generalizes the formula of the third moment of the eigenvalue distribution. 
Let $A_n \sim SW(n,k,p)$ be the adjacency matrix represents the Watts-Strogatz random graph.
\begin{lem}  $$\mathbb{E}[\frac{1}{n} \mathrm{Tr}(A_n^3)] = \frac{1}{n}\sum_{\substack{1 \le i_1,i_2,i_3 \le n }} \mathbb{E}[A_{i_1i_2}A_{i_2i_3}A_{i_3i_1}], \mathrm{where \ }i_1,i_2,i_3 \mathrm{\ distinct}.$$
\end{lem}
\begin{proof}
By (3.7), we have
\begin{eqnarray}
\mathbb{E}[\frac{1}{n} \mathrm{Tr}(A_n^3)] &=& \mathbb{E}[\frac{1}{n}\sum_{1 \le i_1,i_2,i_3 \le n}A_{i_1i_2}A_{i_2i_3}A_{i_3i_1}] \\ &=& \frac{1}{n}\sum_{1 \le i_1,i_2,i_3 \le n} \mathbb{E}[A_{i_1i_2}A_{i_2i_3}A_{i_3i_1}] 
\end{eqnarray}
\noindent We use Proposition 2.3 and the same reasoning in (4.7) about each entry of $A_n$ are either $0$ or $1$ to simplify five possible cases of index values $i_1,i_2,i_3$ from (5.3). \\
1. For $i_1=i_2 \ne i_3$,
\begin{align*}
\mathbb{E}[A_{i_1i_2}A_{i_2i_3}A_{i_3i_1}] &= \mathbb{E}[A_{i_1i_1}A_{i_1i_3}A_{i_3i_1}] = \mathbb{E}[A_{i_1i_1}A_{i_1i_3}A_{i_1i_3}] = \mathbb{E}[A_{i_1i_1}A_{i_1i_3}^2] \\ &= \mathbb{E}[A_{i_1i_1}A_{i_1i_3}].
\end{align*}
2. For $i_1 \ne i_2=i_3$,
\begin{align*}
\mathbb{E}[A_{i_1i_2}A_{i_2i_3}A_{i_3i_1}] &= \mathbb{E}[A_{i_1i_2}A_{i_2i_2}A_{i_2i_1}] = \mathbb{E}[A_{i_1i_2}A_{i_2i_2}A_{i_1i_2}] = \mathbb{E}[A_{i_1i_2}^2A_{i_2i_2}] \\ &= \mathbb{E}[A_{i_1i_2}A_{i_2i_2}]. 
\end{align*}
3. For $i_1=i_3 \ne i_2$,
\begin{align*}
\mathbb{E}[A_{i_1i_2}A_{i_2i_3}A_{i_3i_1}] &= \mathbb{E}[A_{i_1i_2}A_{i_2i_1}A_{i_1i_1}] = \mathbb{E}[A_{i_1i_2}A_{i_1i_2}A_{i_1i_1}] = \mathbb{E}[A_{i_1i_2}^2A_{i_1i_1}] \\ &= \mathbb{E}[A_{i_1i_2}A_{i_1i_1}].
\end{align*}
4. For $i_1=i_2=i_3$,  $\mathbb{E}[A_{i_1i_2}A_{i_2i_3}A_{i_3i_1}] = \mathbb{E}[A_{i_1i_1}A_{i_1i_1}A_{i_1i_1}] = \mathbb{E}[A_{i_1i_1}^3] = \mathbb{E}[A_{i_1i_1}]$. \\
5. For $i_1,i_2,i_3$ are distinct, we keep the same formula, which is $\mathbb{E}[A_{i_1i_2}A_{i_2i_3}A_{i_3i_1}]$. 
\newline \noindent Then the sum in equation (5.3) is factored into five different sums based on the five different conditions of index values $i_1,i_2,i_3$. Hence,
\begin{align*}
    \mathbb{E}[\frac{1}{n} \mathrm{Tr}(A_n^3)] &= \frac{1}{n}\Bigg(\sum_{1 \le i_1,i_3 \le n}\mathbb{E}[A_{i_1i_1}A_{i_1i_3}] + \sum_{1 \le i_1,i_2 \le n}\mathbb{E}[A_{i_1i_2}A_{i_2i_2}] \\ &+ \sum_{1 \le i_1,i_2 \le n}\mathbb{E}[A_{i_1i_2}A_{i_1i_1}]  + \sum_{1\le i_1 \le n}\mathbb{E}[A_{i_1i_1}] \\ &+ \sum_{\substack{1 \le i_1,i_2,i_3 \le n \mathrm{\ distinct}}}  \mathbb{E}[A_{i_1i_2}A_{i_2i_3}A_{i_3i_1}]\Bigg) \\ &= \frac{1}{n}\Bigg(0+0+0+0+ \sum_{\substack{1 \le i_1,i_2,i_3 \le n \mathrm{\ distinct}}}  \mathbb{E}[A_{i_1i_2}A_{i_2i_3}A_{i_3i_1}]\Bigg) \\ &= \frac{1}{n}\sum_{\substack{1 \le i_1,i_2,i_3 \le n \mathrm{\ distinct}}}  \mathbb{E}[A_{i_1i_2}A_{i_2i_3}A_{i_3i_1}]
\end{align*}
The equation holds true because of the same reasoning in (4.2).
\end{proof}
\begin{lem} A generalized formula version of the third moment of the eigenvalue distribution of the small-world random graph is
$$\mathbb{E}[\frac{1}{n} \mathrm{Tr}(A_n^3)] = \frac{1}{n}\sum_{1 \le i_1,i_2,i_3 \le n \mathrm{\ distinct}}  \mathbb{P}(A_{i_1i_2} = 1, A_{i_2i_3} = 1, A_{i_3i_1} = 1).$$
\end{lem}
\begin{proof}
Since the random variable $A_{ij}$ is either $0$ or $1$, it is the Bernoulli distribution. The expectation of the random variable is equal to the probability of the random variable itself. From Lemma 5.1, it follows that
\begin{align*}
    \mathbb{E}[\frac{1}{n} \mathrm{Tr}(A_n^3)] &=
    \frac{1}{n}\sum_{1 \le i_1,i_2,i_3 \le n \mathrm{\ distinct}}  \mathbb{E}[A_{i_1i_2}A_{i_2i_3}A_{i_3i_1}] \\ &=
    \frac{1}{n}\sum_{1 \le i_1,i_2,i_3 \le n  \mathrm{\ distinct}}  \mathbb{P}(A_{i_1i_2} = 1, A_{i_2i_3} = 1, A_{i_3i_1} = 1).
\end{align*}
\end{proof}

The indexes $i_1, i_2, i_3$ within the sum from Lemma 5.4 represent the three distinct vertices $i_1, i_2, i_3$ in $SW(n,k,p)$ random graph. We define new notations to easily understand a vertex relation within the random graph.
\begin{notn}
Given vertices $i,j \in \{1,2,...,n\}$ and $c \in \mathbb{N}$, we define the notation $||i-j|| = c$ is the distance on the torus such that the minimum distance on the circle between vertices $i$ and $j$ is equal to $c$, without considering the direction (upside or downside). In other words, the vertex $j$ is located $c$ vertices apart from the vertex $i$.   For example, we let $n=8$ and $k=4$ with a set of vertices $\{1,2,...,8\}$. Consider a ring lattice of $8$ vertices starting from the vertex $1$ to the vertex $8$, the notation $||i-j|| = 2$, for $i= \mathrm{vertex \ } 1,\ j= \mathrm{vertex \ }7$, means the minimum distance on torus between the vertex $1$ and the vertex $7$ is $2$ apart between two vertices. Alternatively, we can think about if starting from the vertex $1$, we need to jump two steps: first step from vertex $1$ to vertex $8$ and another step from vertex $8$ to reach vertex $7$.
\end{notn}

\begin{notn}
Based on Lemma 5.4, the main sum of the probability is required to have all distinct vertices $i_1,i_2,i_3$ and the cycle of edges $\{i_1,i_2\},\{i_2,i_3\}$ and $\{i_3,i_1\}$. There are four different cases of the vertex's location on the torus that we must recognize the construction of such connected edges. For distinct vertices $i_1,i_2,i_3$ in the graph,\\
\textbf{1.} $||i_1-i_2|| \le \frac{k}{2},||i_2-i_3|| \le \frac{k}{2},||i_3-i_1|| \le \frac{k}{2}$. Each edge is constructed by two vertices where the distance between them is within $\frac{k}{2}$ apart. In the remaining part of the paper, we will call this configuration \textbf{``all close."}\\
\textbf{2.} This case contains two close edges; each is constructed by two vertices where the distance between them is within $\frac{k}{2}$ apart from the other. However, the third edge has two vertices far from each other (the distance apart is more than $\frac{k}{2}$). For any vertex $i_1,i_2,i_3$, those are classified into this case if satisfying one of the possibilities: 
\begin{itemize}
    \item  $||i_1-i_2||>\frac{k}{2},||i_2-i_3||\le \frac{k}{2},||i_3-i_1|| \le \frac{k}{2}$
    \item  $||i_1-i_2||\le \frac{k}{2},||i_2-i_3||>\frac{k}{2},||i_3-i_1|| \le \frac{k}{2}$
    \item  $||i_1-i_2|| \le \frac{k}{2},||i_2-i_3|| \le \frac{k}{2},||i_3-i_1||>\frac{k}{2}$
\end{itemize}
This configuration is called \textbf{``one far."} \\
\textbf{3.} The third configuration is only one edge is constructed by two close vertices (the distance is within $\frac{k}{2}$ apart), while the other two edges have a far distance constructed vertices, where each edge is constructed by two vertices with more than $\frac{k}{2}$ distance apart. Likewise, it follows that
\begin{itemize}
    \item  $||i_1-i_2||>\frac{k}{2},||i_2-i_3||>\frac{k}{2},||i_3-i_1|| \le \frac{k}{2}$
    \item  $||i_1-i_2|| \le \frac{k}{2},||i_2-i_3||>\frac{k}{2},||i_3-i_1||>\frac{k}{2}$
    \item  $||i_1-i_2||>\frac{k}{2},||i_2-i_3||\le \frac{k}{2},||i_3-i_1||>\frac{k}{2}$
\end{itemize}
This configuration is called \textbf{``two far."} \\
\textbf{4.} All three edges are constructed by vertices, where each pair of vertices has the distance more than $\frac{k}{2}$ apart. It follows that $||i_1-i_2||>\frac{k}{2},||i_2-i_3||>\frac{k}{2},||i_3-i_1||>\frac{k}{2}$. This configuration is called \textbf{``all far."}
\end{notn}
\begin{notn}
Let $\mathbf{P_1} =  \mathbb{P}(A_{i_1i_2} = 1, A_{i_2i_3} = 1, A_{i_3i_1} = 1)$ when the vertices $i_1,i_2,i_3$ satisfy \textbf{all close }configuration.
$\mathbf{P_2} =  \max\{\mathbb{P}(A_{i_1i_2} = 1, A_{i_2i_3} = 1, A_{i_3i_1} = 1)\}$ when the vertices $i_1,i_2,i_3$ satisfy \textbf{one far }configuration.  
$\mathbf{P_3} =  \max\{\mathbb{P}(A_{i_1i_2} = 1, A_{i_2i_3} = 1, A_{i_3i_1} = 1)\}$ when the vertices $i_1,i_2,i_3$ satisfy \textbf{two far }configuration. 
$\mathbf{P_4} =  \max\{\mathbb{P}(A_{i_1i_2} = 1, A_{i_2i_3} = 1, A_{i_3i_1} = 1)\}$ when the vertices $i_1,i_2,i_3$ satisfy \textbf{all far }configuration.
\end{notn}

\begin{notn}
Let $\mathbf{C_1}$ is defined to be the cardinality of the set $\{(i_1,i_2,i_3): 1 \le i_1,i_2,i_3 \le n \mathrm{\ distinct \ } \mathrm{and \ } \textbf{all close} \mathrm{ \ configuration}\}$, $\mathbf{C_2}$ is defined to be the cardinality of the set $\{(i_1,i_2,i_3): 1 \le i_1,i_2,i_3 \le n \mathrm{\ distinct \ } \mathrm{and \ } \textbf{one far} \mathrm{ \ configuration}\}$,
$\mathbf{C_3}$ is defined to be the cardinality of the set $\{(i_1,i_2,i_3): 1 \le i_1,i_2,i_3 \le n \mathrm{\ distinct \ } \mathrm{and \ }\\ \textbf{two far} \mathrm{ \ configuration}\}$, $\mathbf{C_4}$ is defined to be the cardinality of the set $\{(i_1,i_2,i_3): 1 \le i_1,i_2,i_3 \le n \mathrm{\ distinct \ } \mathrm{and \ } \textbf{all far} \mathrm{ \ configuration}\}$.
\end{notn}

\begin{notn}
the vertex $i \pm d$ for any $d$ to represent $i \pm d \mathrm{\ (mod \ n)}$.
\end{notn}
\begin{notn}
For any vertex $i,j$ in $SW(n,k,p)$ random graph. We define $i \to j$ is the rewiring from vertex $i$ to vertex $j$. It means that after removing an edge $\{i,i+d\}$ for a particular $d \in \{1,2,...,\frac{k}{2}\}$ with the probability $p$, an edge $\{i,j\}$ is then connected, for some vertex $j$ from randomly choosing from a vertex set $\{1,2,...,n\} \backslash (\{i-\frac{k}{2},...,i-1,i,i+1,...,i+\frac{k}{2}\} \cup \mathbf{N}(i))$. \\ \indent Also, we define $i \xrightarrow{d}  j$ for a specific $d \in \{1,2,...,\frac{k}{2}\}$ is the edge $\{i,i+d\}$ gets rewired to a new edge $\{i,j\}$. In other words, it means, with the probability $p$, the edge $\{i,i+d\}$ gets rewired and be replaced by the edge $\{i,j\}$.
\end{notn}

\begin{lem}
By Notations 5.5-5.10, we have another new generalized version the third moment formula 
    $$\mathbb{E}[\frac{1}{n} \mathrm{Tr}(A_n^3)] = \frac{1}{n}\big[\mathbf{C_1}\mathbf{P_1} + O(\mathbf{C_2}\mathbf{P_2}) + O(\mathbf{C_3}\mathbf{P_3}) + O(\mathbf{C_4}\mathbf{P_4})\big].$$
\end{lem}
\begin{proof}
Without the loss of generality, we consider the bound of all probabilities for each configuration. The \textbf{all close} configuration contains exactly one case when the distance between each pair of two vertices is within $\frac{k}{2}$ apart from each other. The permutation of vertices $i_1,i_2,i_3$ in this configuration gives the same probability $\mathbf{P_1}$ and $\mathbf{C_1}$. However, the permutations for other configurations give different probabilities. We must bound the probabilities for each configuration with the maximum of the probabilities of the vertex permutation in a particular configuration $\mathbf{P_i}$, for $i=2,3,4$. For each configuration, we compute the sum of all probabilities of all vertex permutation by using the bound of the product of the maximum probability $\mathbf{P_i}$ and the number of all permutations $\mathbf{C_i}$. It follows that
\begin{align*}
    \sum_{\mathrm{i \ configuration}}\mathbb{P}(A_{i_1i_2} = 1, A_{i_2i_3} = 1, A_{i_3i_1} = 1) &\le O(\mathbf{C_i} \cdot \mathbf{P_i}),
\end{align*}
where $i=2$ configuration means \textbf{one far} configuration, $i=3$ configuration means \textbf{two far} configuration, and $i=4$ configuration means \textbf{all far} configuration. Thus, by Lemma $5.4$
\begin{align*}
    \mathbb{E}[\frac{1}{n} \mathrm{Tr}(A_n^3)] &= \frac{1}{n}\sum_{1 \le i_1,i_2,i_3 \le n  \mathrm{\ distinct}}  \mathbb{P}(A_{i_1i_2} = 1, A_{i_2i_3} = 1, A_{i_3i_1} = 1) \\ &= \frac{1}{n}\bigg(\sum_{\mathrm{all \ close}} \mathbf{P_1} + \sum_{\mathrm{one \ far}}\mathbb{P}(A_{i_1i_2} = 1, A_{i_2i_3} = 1, A_{i_3i_1} = 1) \\ &+ \sum_{\mathrm{two \ far}}\mathbb{P}(A_{i_1i_2} = 1, A_{i_2i_3} = 1, A_{i_3i_1} = 1) \\ &+ \sum_{\mathrm{all \ far}}\mathbb{P}(A_{i_1i_2} = 1, A_{i_2i_3} = 1, A_{i_3i_1} = 1)\bigg) \\ &= \frac{1}{n}\big[\mathbf{C_1}\mathbf{P_1} + O(\mathbf{C_2}\mathbf{P_2}) + O(\mathbf{C_3}\mathbf{P_3}) + O(\mathbf{C_4}\mathbf{P_4})\big].
\end{align*}
\end{proof}

\section{The Computation of Probabilities}
This section provides the computation of the probabilities $\mathbf{P_1}, \mathbf{P_2}, \mathbf{P_3}$, and $\mathbf{P_4}$. In general, since all permutation of three vertices can rearrange to have a new order of vertices $i_1 < i_2 < i_3$, we will consider only the case that all vertices $i_1,i_2,i_3$ are located orderly in the random graph. Each configuration contains at least one condition. When we assign three vertices $i_1,i_2,i_3$, these will satisfy one of the conditions in four configurations. \\
\indent
Let $i_1,i_2,i_3$ be vertices on the $SW(n,k,p)$ random graph. These vertices are classified as \textbf{all close} configuration. The construction of this configuration follows that
\begin{itemize}
    \item Starting at vertex $i_1$, we need to connect an edge $\{i_1,i_2\}$ such that $||i_1-i_2|| \le \frac{k}{2}$. This edge $(A_{i_1i_2} = 1)$ already exists without the rewiring. So, we keep this edge non-rewiring with the probability $\mathbb{P}(A_{i_1i_2} = 1)= 1-p$.
    \item Then we recognize at vertex $i_2$ and consider an edge $\{i_2,i_3\}$ such that the distance $||i_2-i_3|| \le \frac{k}{2}$. The event $A_{i_2i_3} = 1$ happens if the edge does not rewire. So, we have $\mathbb{P}(A_{i_2i_3}=1)= 1-p$.
    \item Finally, from vertex $i_3$ there is an edge $\{i_3,i_1\}$ with $||i_3-i_1|| \le \frac{k}{2}$ to connect to $i_1$ again. The probability to have this edge is equal to $\mathbb{P}(A_{i_3i_1}=1)= 1-p$.
\end{itemize}
\begin{lem}
    For distinct vertices $i_1,i_2,i_3$ on the $SW(n,k,p)$ random graph such that those vertices satisfy the case \textbf{all close} configuration. The probability
    $\mathbf{P_1} =\mathbb{P}(A_{i_1i_2} = 1, A_{i_2i_3} = 1, A_{i_3i_1} = 1) = (1-p)^3.$
\end{lem}
\begin{proof}
Let the vertices $i_1,i_2,i_3$ be distinct vertices. We need to find the probability that $\{i_1,i_2\},\{i_2,i_3\}$, and  $\{i_3,i_1\}$ are in the random graph. Based on above computation of the probability for the connection of three edges and the independent events of $A_{i_1i_2} = 1,A_{i_2i_3}$ = 1, and $A_{i_3i_1} = 1$ to keep each edge does not rewire, therefore,
$\mathbf{P_1} =\mathbb{P}(A_{i_1i_2} = 1, A_{i_2i_3} = 1, A_{i_3i_1} = 1) = \mathbb{P}(A_{i_1i_2} = 1) \cdot \mathbb{P}(A_{i_2i_3} = 1) \cdot \mathbb{P}( A_{i_3i_1} = 1) = (1-p)^3.$
\end{proof}

\indent  Next, we mainly demonstrates the proof of the probability when the vertices satisfy the case \textbf{one far} configuration. By Lemma $5.11$, we only care about the bound of all probabilities of the vertices in this configuration. We choose the distinct vertices $i_1,i_2,i_3$ in the small-world random graph.
We assume that those vertices satisfy $||i_1-i_2|| \le \frac{k}{2}, ||i_2-i_3|| \le \frac{k}{2}$ and $||i_3-i_1|| > \frac{k}{2}$.  Suppose an edge $\{i_1,i_3\}$ is the only far edge with the distance on the torus between them greater than $\frac{k}{2}$ apart from each other, and the other edges $\{i_1,i_2\},\{i_2,i_3\}$ are constructed by a close distance of any two vertices. We know that there exists two possibilities to rewire and get a new edge $\{i_1,i_3\}$ which are the rewiring from vertex $i_1 \to i_3$ or rewiring from vertex $i_3 \to i_1$. 
\begin{defn}
the notation $\mathbb{P}(i_1 \xrightarrow{d} i_3\ |\ l)$ is the conditional probability of $d^{th}$ downside neighbor of vertex $i_1$ rewires to vertex $i_3$, given that $l$ vertices already rewired to vertex $i_1$.
\end{defn}
\begin{lem} Let $n \in \mathbb{N}$ be an arbitrary number of vertices, $k \in 2\mathbb{N}$ be the degree, and $p \in [0,1]$. For any vertex $i_1,i_3$ in the $SW(n,k,p)$ random graph such that $i_1 < i_3$, those vertices satisfy the condition $||i_3-i_1|| > \frac{k}{2}$. Let $l \le \frac{n}{2}$ be the number of rewirings from some vertices $j < i_1$ to vertex $i_1$. For any $d \in \{1,2,...,\frac{k}{2}\}$, it follows that the probabilities $\mathbb{P}(i_1 \xrightarrow{d} i_3\ |\ l) = O_{k}(\frac{1}{n})$ and $\mathbb{P}(i_3 \xrightarrow{d} i_1\ |\ l) = O_{k}(\frac{1}{n})$.
\end{lem}
\begin{proof} Consider the probability of rewiring from $i_1 \to i_3$, we let there exist $l$ vertices already rewired to vertex $i_1$. In this proof, we only care the case $l \le \frac{n}{2}$.
We know that the vertex $i_1$ contains  $\frac{k}{2}$ downside edges. Due to the rewiring process, each edge $\{i_1,i_1+d\}$ for $d= \{1,...,\frac{k}{2}\}$ could possibly be replaced by the edge $\{i_1,i_3\}$. 
\begin{itemize}
    \item $d=1$, the edge $\{i_3,i_3+1\}$ is rewired with the probability $p$ and there exists $n-k-1-l$ (not vertex $i_3$, other $k$ neighborhood edges, and $l$ previous rewirings) choices for uniformly choosing vertex $i_1$ at random. We know $N(i_3) = \{i_3-\frac{k}{2},...,i_3-1,i_3+1,...,i_3+\frac{k}{2} \mathrm{\ (mod \ n)}\}$ since no edges $(i_3,i_3+v')$ for $v'= 1,...,\frac{k}{2}$ get rewired yet. Thus, 
\end{itemize}
$$\mathbb{P}(i_3 \xrightarrow{1} i_1 |\ l) = \frac{p}{n-k-1-l}$$ 
\begin{itemize}
    \item $d=2$, the edge $\{i_3,i_3+2\}$ is rewired with the probability $p$. There two cases to consider whether or not the previous edge $\{i_3,i_3+1\}$ is rewired to vertex not $i_1$.
\end{itemize}
\begin{align*}
  \mathbb{P}(i_3 \xrightarrow{2} i_1|\ l) &= \mathbb{P}(i_3 \xrightarrow{2} i_1, \text{but } \{i_3,i_3+1\} \ \mathrm{non-rewiring}) \\ &+ \mathbb{P}(i_3 \xrightarrow{2} i_1, \text{but } i_3+1 \not\to i_1).
\end{align*}
\noindent With the probability $1-p$, we consider the edge $\{i_3,i_3+1\}$ is non-rewiring. Then $\{i_3,i_3+2\}$ rewires with the probability $p$ to vertex $i_1$ with $n-k-1-l$ choices uniformly choosing at random. In addition, if $\{i_3,i_3+1\}$ is already rewired with the probability $p$ to some vertex not $i_1$, there are $n-k-2-l$ choices (not $i_1$, its neighbors, and $l$ previous rewirings) out of $n-k-1-l$ to uniformly be chosen. Finally, $\{i_3,i_3+2\}$ is rewired to vertex $i_1$ with $n-k-2-l$ choices left. Thus,
\begin{align*}
    \mathbb{P}(i_3 \xrightarrow{2} i_1|\ l) &= (1-p)\cdot\frac{p}{n-k-1-l} + \frac{(n-k-2-l)p}{n-k-1-l} \cdot \frac{p}{n-k-2-l} \\ &= \frac{p}{n-k-1-l}\cdot(1-p+p) 
\end{align*}
The last equality holds by the simplification. \\
    \indent Let $d \in \{1,2,...,\frac{k}{2}\}$, there are $d$ different cases to consider. We start with all edges $\{i_1,i_1+1\},\{i_1,i_1+2\},...,\{i_1,i_1+(d-1)\}$ that do not  rewire with the probability $(1-p)^{d-1}$. Then an edge $\{i_1,i_1+d\}$ gets rewired to $i_3$ by uniformly choosing $n-k-1-l$ choices (not $i_1$, its neighbors, and other previous $l$ vertices). In the second case, we have $\binom{d-1}{1}$ ways to pick one edge from $\{i_1,i_1+1\},\{i_1,i_1+2\},...,\{i_1,i_1+(d-1)\}$ to rewire with the probability $p$ to vertex not $i_3$ by choosing $n-k-2-l$ choices out of $n-k-1-l$. We keep the remaining edges non-rewiring with the probability $(1-p)^{d-2}$ before $\{i_1,i_1+d\}$ is rewired to $i_3$ by uniformly choosing $n-k-2-l$ choices (not the first rewiring vertex, its neighbors, and other previous $l$ vertices). The third step begins with $\binom{d-1}{2}$ ways to pick two edges from $\{i_1,i_1+1\},\{i_1,i_1+2\},...,\{i_1,i_1+(d-1)\}$ to be rewired. The first chosen edge gets rewired by choosing a random vertex not $i_3$ and previous $l$ vertices with $n-k-2-l$ choices out of $n-k-1-l$, and the second one gets rewired by choosing another random vertex with $n-k-3-l$ choices (not $i_3$, its neighbors, the first rewiring vertex, and previous $l$ vertices) out of $n-k-2-l$. We keep the remaining edges non-rewiring with the probability $(1-p)^{d-3}$, and then $\{i_1,i_1+d\}$ is rewired by uniformly choosing $i_3$ from the remaining $n-k-3-l$ choices. It continues the same procedure for computing the probability until all chosen $(d-1)$ edges from $\{i_1,i_1+1\},\{i_1,i_1+2\},...,\{i_1,i_1+(d-1)\}$ get rewired. Finally, the last edge $\{i_1,i_1+d\}$ is rewired with the probability $p$ by uniformly choosing vertex $i_3$ from the remaining $n-k-\frac{k}{2}-l$ choices. Therefore, we have the conditional probability
Then, we consider each probability. \\
$\mathbb{P}(i_1 \xrightarrow{d} i_3, \text{but } \{i_1,i_1+1\},\{i_1,i_1+2\},...,\{i_1,i_1+(d-1)\} \text{ non-rewiring})$, \\
$\mathbb{P}(i_1 \xrightarrow{d} i_3, \text{but only one edge rewires to not } i_3)$
\begin{align*}
  &\mathbb{P}(i_1 \xrightarrow{d} i_3 |\ l) \\ &=  \mathbb{P}(i_1 \xrightarrow{d} i_3, \text{but } \{i_1,i_1+1\},\{i_1,i_1+2\},...,\{i_1,i_1+(d-1)\} \text{ non-rewiring}) \\ &+ \mathbb{P}(i_1 \xrightarrow{d} i_3, \text{but only one edge rewires to not } i_3) \\ &+ \mathbb{P}(i_1 \xrightarrow{d} i_3, \text{but two edges rewire to not }i_3) + ... + \\ &+ \mathbb{P}(i_1 \xrightarrow{d} i_3, \text{but all } (d-2) \text{ edges rewire to not }i_3)\\ &+ \mathbb{P}(i_1 \xrightarrow{d} i_3, \text{but }\{i_1,i_1+1\},\{i_1,i_1+2\},...,\{i_1,i_1+(d-1)\} \text{ rewire to not } i_3)
\end{align*} 
Then, 
\begin{align*}
  &= \binom{d-1}{0}(1-p)^{d-1}\cdot\frac{p}{n-k-1-l} \\ &+ \binom{d-1}{1}(1-p)^{d-2}\cdot \frac{(n-k-2-l)p}{n-k-1-l}\cdot\frac{p}{n-k-2-l} \\ &+ \binom{d-1}{2}(1-p)^{d-3}\cdot\frac{(n-k-2-l)p}{n-k-1-l}\cdot\frac{(n-k-3-l)p}{n-k-2-l}\cdot\frac{p}{n-k-3-l} \ + ... + \\ &+ \binom{d-1}{d-2}(1-p)\cdot\frac{(n-k-2-l)p}{n-k-1-l}\cdot\frac{(n-k-3-l)p}{n-k-2-l}\cdot\cdot\cdot \\ &\cdot\cdot\cdot \frac{(n-k-(\frac{k}{2}-1)-l)p}{n-k-(\frac{k}{2}-2)-l}\cdot\frac{p}{n-k-(\frac{k}{2}-1)-l} \\ &+ \binom{d-1}{d-1}\frac{(n-k-2-l)p}{n-k-1-l}\cdot\frac{(n-k-3-l)p}{n-k-2-l}\cdot\cdot\cdot \frac{(n-k-(\frac{k}{2})-l)p}{n-k-(\frac{k}{2}-1)-l}\cdot\frac{p}{n-k-\frac{k}{2}-l} \\ &= \left(\frac{p}{n-k-1-l}\right) \Bigg[\binom{d-1}{0}(1-p)^{d-1} + \binom{d-1}{1}(1-p)^{d-2}p \\ &+ \binom{d-1}{2}(1-p)^{d-3}p^2 +...+ \binom{d-1}{d-2}(1-p)p^{d-2} + \binom{d-1}{d-1}p^{d-1}\Bigg] \\ &= \left(\frac{p}{n-k-1-l}\right)\Bigg(\sum_{j=0}^{d-1}\binom{d-1}{j}(1-p)^{d-1-j}p^j \Bigg) \\ &\le \left(\frac{1}{n-k-1-l}\right)\Bigg(\sum_{j=0}^{d-1}\binom{d-1}{j}(1-p)^{d-1-j}p^j \Bigg) \\ &= O_{k}(\frac{1}{n}), \mathrm{\ since\ }l \le \frac{n}{2}.
\end{align*}
Since the rewiring $i_3 \xrightarrow{d} i_1$ given that there exist some $l \le \frac{n}{2}$ previous edges rewired to vertex $i_3$ (not $i_1$ itself), it can be done by rewiring from one of $i_3$'s downside neighbors to some vertex choosing uniformly with the constraint $l$. Since we relax the number of the previous rewirings to $i_3$ with the extreme range of $l \le \frac{n}{2}$, the computation can exclude the case that there exists the rewiring $i_1 \to i_3$ by the time the vertex $i_3$ is considered. Hence, it follows the same computation as the rewiring from $i_1 \xrightarrow{d} i_3$. The probability has the same bound which is
$\mathbb{P}(i_3 \xrightarrow{d} i_1 |\ l) = O_{k}(\frac{1}{n})$.
\end{proof}
\begin{lem}
Let $i_1,i_3$ be vertices in the $SW(n,k,p)$ random graph which $i_1 < i_3$ and $||i_1-i_3|| > \frac{k}{2}$. The probabilities $\mathbb{P}(i_1 \to i_3) = O_{k}(\frac{1}{n})$ and $\mathbb{P}(i_3 \to i_1) = O_{k}(\frac{1}{n})$.
\end{lem}
\begin{proof}
Let $i_1,i_3$ be distinct vertices on the $SW(n,k,p)$ random graph. We will consider the case $i_1 \to i_3$, and then we will use the same computation to come up with the probability of $i_3 \to i_1$. 
Let $l$ be the number of previous rewirings to vertex $i_1$. In this proof, we try to avoid some complicated computation by having a bound of $0 \le l \le n$. We consider 
\begin{align*}
    \mathbb{P}(i_1 \xrightarrow{d} i_3) &= \sum_{t=0}^{n}\mathbb{P}(i_1 \xrightarrow{d} i_3 |\ l=t)\cdot \mathbb{P}(l=t) \\ &= \sum_{t=0}^{\frac{n}{2}}\mathbb{P}(i_1 \xrightarrow{d} i_3 |\ l=t)\cdot \mathbb{P}(l=t) + \sum_{t=\frac{n}{2}+1}^{n}\mathbb{P}(i_1 \xrightarrow{d} i_3 |\ l=t)\cdot \mathbb{P}(l=t).
\end{align*}
By Lemma $6.3$, when $0 \le l \le \frac{n}{2}$, the probability $\mathbb{P}(i_1 \xrightarrow{d} i_3 |\ l=t)$ is bounded by $O_{k}(\frac{1}{n})$. It makes the term $\sum_{t=0}^{\frac{n}{2}}\mathbb{P}(i_1 \xrightarrow{d} i_3 |\ l=t)\cdot \mathbb{P}(l=t)$ have the same bound. However, if $l \ge \frac{n}{2}+1$, the second term $\sum_{t=\frac{n}{2}+1}^{n}\mathbb{P}(i_1 \xrightarrow{d} i_3 |\ l=t)\cdot \mathbb{P}(l=t)$ will be bounded by the probability $\mathbb{P}(l \ge \frac{n}{2}+1)$. It follows that
\begin{align*}
    \mathbb{P}(i_1 \xrightarrow{d} i_3) &\le \sum_{t=0}^{\frac{n}{2}}\mathbb{P}(i_1 \xrightarrow{d} i_3 |\ l=t)\cdot \mathbb{P}(l=t) + \sum_{t=\frac{n}{2}+1}^{n}\mathbb{P}(i_1 \xrightarrow{d} i_3 |\ l=t)\cdot \mathbb{P}(l=t) \\ &\le O_{k}(\frac{1}{n}) + \sum_{t=\frac{n}{2}+1}^{n} \mathbb{P}(l=t).
\end{align*}  
\indent In addition, the probability $\mathbb{P}(l=t)$ is computed by a bound. We start computing the combination of choosing $l=t$ options from $n$ options to rewire to vertex $i_1$ before this vertex is considered in the rewiring process. If we have $l \ge \frac{n}{2}+1$, the vertex $i_1$  may be close to vertex $n$. The rewiring algorithm does not allow to choose a new vertex that lies within $k$ neighbors. Hence, the closest vertex $j$ that can be rewired to vertex $i_1$ cannot be too close to $i_1$. In order to simplify the computation, we ignore all upside $\frac{n}{4}$ neighbors of $i_1$. \\ 
\indent Since we assume that $l \ge \frac{n}{2}+1$, we must carefully consider the proper bound of the probability. In order to have the bound, we need to compute the worst case of location of vertex $i_1$ for some number of $l$. Since we ignore all $\frac{n}{4}$ upside neighbors of $i_1$, we have at least $\frac{n}{4}$ all connections to $i_1$. Each rewiring to vertex $i_1$ has the probability $\frac{1}{n-k-1-l}$ for the number $l \le \frac{n}{2}$ vertices already rewired to vertex $i_1$. To compute a bound of this fraction, we know that there exists some number $c \in \mathbb{R}$ such that $\frac{1}{n-k-1-l} \le \frac{c}{n}$. Since there are at least $ \frac{n}{4}$ vertices rewire to vertex $i_1$ and the rewirings are independent, it follows that
\begin{align*}
    \mathbb{P}(l=t) \le \binom{n}{t}\cdot \big(\frac{c}{n}\big)^{\frac{n}{4}}.
\end{align*}
Hence,
\begin{align*}
    \mathbb{P}(i_1 \xrightarrow{d} i_3) &\le O_{k}(\frac{1}{n}) + \sum_{t=\frac{n}{2}+1}^{n}\binom{n}{t} \cdot \big(\frac{c}{n}\big)^{\frac{n}{4}}.
\end{align*}
By the Binomial Theorem, it follows that
\begin{align*}
    \mathbb{P}(i_1 \xrightarrow{d} i_3) &\le O_{k}(\frac{1}{n}) + 2^n \cdot \big(\frac{c}{n}\big)^{\frac{n}{4}} \le O_{k}(\frac{1}{n}) + O_{k}(\frac{1}{n^3}) \le O_{k}(\frac{1}{n}).
\end{align*}
For $1 \le d \le \frac{k}{2}$, a particular downside $d^{th}$ neighbor of $i_1$ can rewire to vertex $i_3$. Thus, we have the probability 
\begin{align*}
    \mathbb{P}(i_1 \to i_3) &= \sum_{d=1}^{\frac{k}{2}}\mathbb{P}(i_1 \xrightarrow{d} i_3) = (\frac{k}{2}) \cdot O_{k}(\frac{1}{n}) \le O_{k}(\frac{1}{n}).
\end{align*}
Similarly, we use a bound of the probability of rewiring given that $0 \le l \le n$. The rewiring $i_3 \to i_1$, follows the same computation as $i_1 \to i_3$. Thus we have the probability $\mathbb{P}(i_3 \to i_1) = O_{k}(\frac{1}{n})$ as well.
\end{proof}
\begin{lem}
   Let $i_1,i_2,i_3$ be distinct vertices in  $SW(n,k,p)$ random graph. Suppose three vertices satisfy the condition $||i_1-i_2|| \le \frac{k}{2}$ and $||i_1-i_3|| > \frac{k}{2}$, then it follows that the probability  $\mathbb{P}(i_1 \to i_3, \text{given an edge } \{i_1,i_2\} \text{ non-rewiring}) = O_{k}(\frac{1}{n}).$
\end{lem}
\begin{proof}
To compute the probability of rewiring from $i_1 \to i_3$ but given the edge $\{i_1,i_2\}$ is non-rewiring, we consider that there exist $\frac{k}{2}-1$ edges out of $\frac{k}{2}$ to possibly be rewired to $i_3$ because we need to keep one edge $\{i_1,i_2\} = \{i_1,v'\}$ non-rewiring for a chosen vertex $v' \in \{i_1+1,...,i_1+\frac{k}{2}\}$. Let $l$ be the number of vertices that rewired to vertex $i_1$.
With a similar computation from Lemma $6.3$, if one of the downside neighborhood edges of vertex $i_1$ (includes an edge $\{i_1,i_2\}$) can be rewired to vertex $i_3$ with a far distance on the torus between $i_1,i_3$ $(> \frac{k}{2})$, the conditional probability $\mathbb{P}(i_1 \xrightarrow{d} i_3 |\ l) = O_k (\frac{1}{n})$. By Lemma $6.4$, the probability $\mathbb{P}(i_1 \xrightarrow{d} i_3) = O_{k}(\frac{1}{n})$. Since not all downside neighborhood edges of $i_1$ have a chance to be rewired to vertex $i_3$ (need to keep an edge $\{i_1,i_2\}$ non-rewiring), it comes up with a smaller probability of the rewiring $i_1 \xrightarrow{d} i_3$. It follows that
\begin{align*}
    \mathbb{P}(i_1 \to i_3, \text{given } \{i_1,i_2\} \text{ non-rewiring}) &= \sum_{d=1}^{\frac{k}{2}} \mathbb{P}(i_1 \xrightarrow{d} i_3 |\ \{i_1,i_2\} \text{ non-rewiring})  \\ &\le \sum_{d=1}^{\frac{k}{2}}\mathbb{P}(i_1 \xrightarrow{d} i_3) \le O_{k}(\frac{1}{n}).
\end{align*}
\end{proof}
\begin{lem} Given the case \textbf{one far} configuration. Let $i_1,i_2,i_3$ be distinct vertices on the $SW(n,k,p)$ random graph satisfies one of the following three conditions;
\begin{enumerate}
    \item  $||i_1-i_2|| \le \frac{k}{2},||i_2-i_3|| \le \frac{k}{2},||i_3-i_1|| > \frac{k}{2}$
    \item  $||i_1-i_2|| \le \frac{k}{2},||i_2-i_3||>\frac{k}{2},||i_3-i_1|| \le \frac{k}{2}$
    \item  $||i_1-i_2|| > \frac{k}{2},||i_2-i_3|| \le \frac{k}{2},||i_3-i_1|| \le \frac{k}{2}$
\end{enumerate}
Then the probability $\mathbf{P}_2 = O_{k}(\frac{1}{n})$.
\end{lem}
\begin{proof}
 Let $\mathbf{P}_{2,i}$ be a maximum probability of the above condition $i$ for \textbf{one far} configuration. We will prove the probability bound  $\mathbf{P}_{2,1}$ and then use the result to come up with others probabilities $\mathbf{P}_{2,2}$ and $\mathbf{P}_{2,3}$.
 First, we choose the vertices $i_1,i_2,i_3$
 that satisfy $||i_1-i_2|| \le \frac{k}{2},||i_2-i_3|| \le \frac{k}{2},||i_3-i_1|| > \frac{k}{2}$. 
 We assume those vertices give a maximum probability of the first condition of \textbf{one far} configuration. We keep two edges $\{i_1,i_2\}, \{i_2,i_3\}$ non-rewiring and rewire an edge from either $i_1$ neighbor or $i_3$ neighbor to a new edge $\{i_1,i_3\}$. To construct the edge $\{i_1,i_3\}$, we start with two possibilities for rewiring, which are the rewiring $i_1 \to i_3$ or $i_3 \to i_1$. We know that the probabilities will be different depending on where the vertices are in the small-world random graph. It is easier for this computation because we will use a bound for the probability. Hence,
\begin{align*}
    \mathbf{P}_{2,1} &= \mathbb{P}(A_{i_1i_2}= 1, A_{i_2i_3} = 1, A_{i_3i_1} = 1)\\ &= \mathbb{P}(A_{i_1i_2}= 1, A_{i_2i_3} = 1, i_1 \to i_3) + \mathbb{P}(A_{i_1i_2}= 1, A_{i_2i_3} = 1, i_3 \to i_1)
\end{align*}
The event $A_{i_2i_3}=1$ is independent from $A_{i_1i_2}=1$ and $i_1 \to i_3$, and the event $i_3 \to i_1$ is independent from $A_{i_1i_2}=1, A_{i_2i_3}=1$. Thus,
\begin{align*}
    \mathbf{P}_{2,1} &= \mathbb{P}(A_{i_2i_3} = 1)\cdot \mathbb{P}(A_{i_1i_2} = 1, i_1 \to i_3) + \mathbb{P}(A_{i_1i_2} = 1, A_{i_2i_3} = 1)\cdot \mathbb{P}(i_3 \to i_1) \\ &= \mathbb{P}(A_{i_2i_3} = 1)\cdot \mathbb{P}(i_1 \to i_3 | A_{i_1i_2} = 1) \cdot \mathbb{P}(A_{i_1i_2} = 1) \\ &+ \mathbb{P}(A_{i_1i_2} = 1, A_{i_2i_3} = 1) \cdot \mathbb{P}(i_3 \to i_1)
\end{align*}
By Lemmas 6.4 and 6.5, the probability 
    $$\mathbf{P}_{2,1} \le (1-p)^2\cdot O_{k}(\frac{1}{n}) + (1-p)^2\cdot O_{k}(\frac{1}{n}) \le (2)O_{k}(\frac{1}{n}) = O_{k}(\frac{1}{n}).$$
\indent For the second condition, $\{i_2,i_3\}$ is the only far edge that can be rewired from either vertex $i_2$ or $i_3$.
This gives us two close edges $\{i_1,i_2\}, \{i_3,i_1\}$ with probability $(1-p)^2$. In this situation, we can only consider the probability for rewiring of far edge $\{i_2,i_3\}$.
By Lemma $6.4$, since $||i_2-i_3|| > \frac{k}{2}$, the probability bound is $\mathbb{P}(i_2 \to i_3) = O_{k}(\frac{1}{n})$. Moreover, we know that $||i_3-i_1|| \le \frac{k}{2}$ and $||i_3-i_2|| > \frac{k}{2}$. We can conclude from Lemma $6.5$ that the probability $\mathbb{P}(i_3 \to i_2 | \{i_3,i_1\}\text{\ non-rewiring}) = O_{k}(\frac{1}{n})$. Thus, the probability $\mathbf{P}_{2,2} = \max\{\mathbb{P}(A_{i_1i_2}= 1, A_{i_2i_3} = 1, A_{i_3i_1} = 1)\ | \ i_1,i_2,i_3 \mathrm{\ satisfy\ second\ condition}\} = O_{k}(\frac{1}{n})$. \\ \indent The third condition follows the same computation as the first and second. We have $||i_1,i_2|| > \frac{k}{2}$, $||i_2-i_3|| \le \frac{k}{2}$, and $||i_3-i_1|| \le \frac{k}{2}$. By Lemmas $6.4$ and $6.5$, it shows that the probability $\mathbf{P}_{2,3} = \max\{\mathbb{P}(A_{i_1i_2}= 1, A_{i_2i_3} = 1, A_{i_3i_1} = 1) \ | \ i_1,i_2,i_3 \mathrm{\ satisfy\ third\ condition}\} = O_{k}(\frac{1}{n})$. Since three probabilities have the same bound, we conclude that the probability of \textbf{one far} configuration is $\mathbf{P}_2 = \max\{\mathbf{P}_{2,1}, \mathbf{P}_{2,2}, \mathbf{P}_{2,3}\} = O_{k}(\frac{1}{n})$.
\end{proof}
\indent The next part illustrates the proof of the case \textbf{two far} configuration. We mainly focus on the condition $||i_1-i_2|| \le \frac{k}{2}, ||i_2-i_3|| > \frac{k}{2}$ and $||i_3-i_1|| > \frac{k}{2}$ for distinct vertices $1 \le i_1 < i_2 < i_3 \le n$ in the $SW(n,k,p)$ random graph. In addition, we use the same computing idea from this condition to prove that the same bound holds with all three conditions. \\ \indent We assume that the edge $\{i_1,i_2\}$ is the only close edge being constructed by the vertices $i_1,i_2$ with a close distance on the torus $(\le \frac{k}{2} \ \mathrm{apart})$. The other edges  $\{i_2,i_3\}$ and $\{i_3,i_1\}$ are constructed by two far distance vertices $(> \frac{k}{2} \ \mathrm{apart})$. This section provides two lemmas about the rewiring either from vertex $i_3 \to i_1$ or $i_3 \to i_2$ separately, and the rewiring from vertex $i_3 \to i_1$ and $i_3 \to i_2$ together after the random graph is created. 
\begin{lem}
    Let $n \in \mathbb{N}$ be an arbitrary number of vertices, $k \in 2\mathbb{N}$ be the degree, and $p \in [0,1]$. For any distinct vertex $i_1,i_2,i_3$ in the $SW(n,k,p)$ random graph, those vertices satisfy the condition $||i_2-i_3|| > \frac{k}{2}$ and $||i_3-i_1|| > \frac{k}{2}$. Let $l \le \frac{n}{2}$ be the number of previous rewiring edges to vertex $i_3$. For any $d \in \{1,2,...,\frac{k}{2}\}$, it follows that the probabilities
    $\mathbb{P}(i_3 \xrightarrow{d} i_1 |\ l) = O_{k}(\frac{1}{n})$ and $\mathbb{P}(i_3 \xrightarrow{d} i_2 |\ l) = O_{k}(\frac{1}{n})$.
\end{lem}
\begin{proof} In this proof, we compute only the case $i_3 \to i_1$. We will show that even though the permutation of vertices gives a different probability, each has the same bound of the probability. The rewiring $i_3 \to i_1$ can be done by a downside edge $d$ of $i_3$ neighbors for $d \in \{1,2,...,\frac{k}{2}\}$.
We start the proof with $d = 1,2$ to demonstrate the pattern of the general term $d$ that will be shown in the last part of this proof. Thus, 
\begin{itemize}
    \item d = 1; with the probability $p$, the edge $\{i_3,i_3+1\}$ is rewired. We know that no other edges on the downside of vertex $i_3$ get rewired yet. There exist $n-k-1-l$ choices (not vertex $i_3$, other $k$ neighbors, and $l$ previous rewirings) for uniformly choosing vertex $i_1$ at random. Since it can possibly rewire to vertex $i_2$, we must eliminate the option of choosing vertex $i_2$. We only have $n-k-2-l$ choices left. Hence,
    \begin{align*}
        \mathbb{P}(i_3 \xrightarrow{1} i_1 |\ l) &= \frac{p}{n-k-2-l}.
    \end{align*}
    \item 
    d = 2; we divide the computing to two cases, which are the edge $\{i_3,i_3+1\}$ is non-rewiring and was already rewired to some vertex not $i_1$.
    \begin{align*}
    \mathbb{P}(i_3 \xrightarrow{2} i_1 |\ l) &= \mathbb{P}(i_3 \xrightarrow{2} i_1, \text{but } \{i_3,i_3+1\} \ \mathrm{non-rewiring}) \\ &+ \mathbb{P}(i_3 \xrightarrow{2} i_1, \text{but } i_3+1 \not\to i_1).
    \end{align*}
    In the first case, an edge $\{i_3,i_3+1\}$ is no rewiring with the probability $1-p$. Then an edge $\{i_3,i_3+2\}$ gets rewired with the probability $p$, and there exist $n-k-2-l$ choices (not $i_3,i_2$, other $k$ neighbors, and $l$ previous rewirings) for uniformly choosing vertex $i_1$ at random. Second, an edge $\{i_3,i_3+1\}$ was already rewired to some vertex not $i_1$ with the rewiring probability $p$. It uniformly chooses some vertex \\ $j \in \{1,2,...,n\} \backslash \bigg(\{i_3-\frac{k}{2},...,i_3-1,i_3,i_3+1,...,i_3+\frac{k}{2}\} \cup \{i_1,i_2\} \bigg)$ at random with the probability $\frac{n-k-3}{n-k-1}$. Then, an edge $\{i_3,i_3+2\}$ rewires with the probability $p$, and there exist $n-k-3$ choices such that $$i_1 \in \{1,2,...,n\} \backslash \left(\{i_3-\frac{k}{2},...,i_3-1,i_3,i_3+1,...,i_3+\frac{k}{2}\} \cup \{i_2,j\} \right)$$ for uniformly choosing vertex $i_1$ at random. Thus,
   \begin{align*}
   \mathbb{P}(i_3 \xrightarrow{2} i_1 |\ l) &= \binom{1}{0}(1-p)\cdot \frac{p}{n-k-2-l} \\ &+ \binom{1}{1}\frac{(n-k-3-l)p}{n-k-1-l}\cdot \frac{p}{n-k-3-l} \\ &= \binom{1}{0}(1-p)\cdot \frac{p}{n-k-2-l}+ \binom{1}{1}\frac{p^2}{n-k-1-l}.
   \end{align*}
    Similarly, for the general $d \in \{1,2,...,\frac{k}{2}\}$. We assume that an edge $\{i_3,i_3+d\}$ gets rewired to $i_1$. It is divided into $d$ different possible cases. First, all previous $d-1$ edges $(\{i_3,i_3+1\}, \{i_3,i_3+2\},..., \{i_3,i_3+(d-1)\} )$ are non-rewiring with the probability $(1-p)^{d-1}$, and an edge $\{i_3,i_3 + d \}$ rewires with the probability $p$ and there exist $n-k-2-l$ choices (not $i_2$, other $k$ neighbors, and $l$ previous rewirings to $i_3$) to uniformly choose vertex $i_1$ at random. Then, we consider the case that the graph has only one previous edge from $\{ \{i_3,i_3+1\}, \{i_3,i_3+2\},..., \{i_3,i_3+(d-1)\} \}$ to be rewired to vertex not $i_1$. With the probability $p$, a given edge rewires and uniformly chooses a new vertex with the probability $\frac{n-k-3-l}{n-k-1-l}$ (not $i_1,i_2$, other $k$ neighbors, and $l$ previous rewirings to $i_3$). Then other $d-2$ remaining edges are non-rewiring with the probability $(1-p)^{d-2}$. Afterward, the last edge $\{i_3,i_3+d\}$ rewires with the probability $p$ and uniformly choosing vertex $i_1 \in \{1,2,...,n\} \backslash \left(\{i_3-\frac{k}{2},...,i_3-1,i_3,i_3+1,...,i_3+\frac{k}{2}\} \cup \{i_2\} \cup \mathbf{N}(i_3) \right)$ with $n-k-3-l$ choices. Another possible case is when there exist two previous edges from $\{ \{i_3,i_3+1\}, \{i_3,i_3+2\},..., \{i_3,i_3+(d-1)\} \}$ rewire to some vertices not $i_1$. We start with the first rewiring edge. There are $n-k-3-l$ choices (not $i_1,i_2$, other $k$ neighbors, and $l$ previous rewirings to $i_3$) out of $n-k-1-l$ for choosing vertex $v_1$. The second edge rewires and uniformly chooses vertex $v_2$ with $n-k-4-l$ choices (not $i_1, i_2, v_1$ , other $k$ neighbors, and $l$ previous rewirings to $i_3$) out of $n-k-2-l$. Finally, an edge $\{i_3,i_3+d\}$ gets rewired with the probability $p$ and uniformly choosing vertex $i_1 \in \{1,2,...,n\} \backslash \left(\{i_3-\frac{k}{2},...,i_3-1,i_3,i_3+1,...,i_3+\frac{k}{2} \} \cup \{i_2, v_1, v_2\} \right)$ with $n-k-4-l$ choices. We continue this computation until all $d$ rewiring edges are considered. For $l \le \frac{n}{2}$, it follows that
\end{itemize}
\begin{align*}
    &\mathbb{P}(i_3 \xrightarrow{d} i_1 |\ l) \\ &= \mathbb{P}(i_3 \xrightarrow{d} i_1, \text{but } \{i_3,i_3+1\},\{i_3,i_3+2\},...,\{i_3,i_3+(d-1)\} \text{ non-rewiring}) \\ &+ \mathbb{P}(i_3 \xrightarrow{d} i_1, \text{but only one edge rewires to not } i_1) \\ &+ \mathbb{P}(i_3 \xrightarrow{d} i_1, \text{but two edges rewire to not }i_1) + ... + \\ &+ \mathbb{P}(i_3 \xrightarrow{d} i_1, \text{but all } (d-2) \text{ edges rewire to not }i_1)\\ &+ \mathbb{P}(i_3 \xrightarrow{d} i_1, \text{but }\{i_3,i_3+1\},\{i_3,i_3+2\},...,\{i_3,i_3+(d-1)\} \text{ rewire to not } i_1) \\ &= \binom{d-1}{0}(1-p)^{d-1}\cdot \frac{p}{n-k-2-l} \\ &+ \binom{d-1}{1}(1-p)^{d-2}\cdot \frac{(n-k-3-l)p}{n-k-1-l}\cdot\frac{p}{n-k-3-l} \\ &+ \binom{d-1}{2}(1-p)^{d-3}\cdot\frac{(n-k-3-l)p}{n-k-1-l}\cdot\frac{(n-k-4-l)p}{n-k-2-l}\cdot\frac{p}{n-k-4-l} \\ &+ \binom{d-1}{3}(1-p)^{d-4}\cdot\frac{(n-k-3-l)p}{n-k-1-l}\cdot\frac{(n-k-4-l)p}{n-k-2-l}\cdot\frac{(n-k-5-l)p}{n-k-3-l} \cdot \\ &\cdot \frac{p}{n-k-5-l} +...+ \\ &+ \binom{d-1}{d-2}(1-p)\cdot\frac{(n-k-3-l)p}{n-k-1-l}\cdot\frac{(n-k-4-l)p}{n-k-2-l}\cdot \cdot \cdot \\ &\cdot \frac{(n-k-(d)-l)}{n-k-(d-2)-l}\cdot\frac{p}{n-k-(d)-l} \\ &+ \binom{d-1}{d-1}\frac{(n-k-3-l)p}{n-k-1-l}\cdot\frac{(n-k-4-l)p}{n-k-2-l}\cdot \cdot \cdot \\ &\cdot \frac{(n-k-(d+1)-l)}{n-k-(d-1)-l}\cdot \frac{p}{n-k-(d+1)-l} \\ &= \binom{d-1}{0}(1-p)^{d-1}\cdot \frac{p}{n-k-2-l} + \binom{d-1}{1}(1-p)^{d-2}\cdot \frac{p^2}{n-k-1-l} \\ &+ \sum_{i=2}^{d-1}\binom{d-1}{i}(1-p)^{d-1-i}\cdot \frac{[n-k-(i+1)-l]}{(n-k-1-l)(n-k-2-l)}\cdot p^{i+1}. \\ &\le \frac{1}{n-k-2-l} + \binom{d-1}{1}\frac{1}{n-k-1-l} + \left(\frac{1}{n-k-2-l}\right)\sum_{i=2}^{d-1}\binom{d-1}{i} \\ &= O_{k}(\frac{1}{n}).
\end{align*}
Since we know $||i_3-i_2|| > \frac{k}{2}$ and $i_2 < i_3$, it follows the same computation as the rewiring $i_3 \to i_1$. Thus, the probability  $\mathbb{P}(i_3 \xrightarrow{d} i_2 |\ l) = O_{k}(\frac{1}{n})$.
\end{proof} 
\begin{lem} 
Let $i_1,i_2,i_3$ be the vertices on the $SW(n,k,p)$ random graph which is $i_1 < i_2 < i_3$. The vertices must satisfy the condition $||i_3-i_1|| > \frac{k}{2}$ and $||i_3-i_2|| > \frac{k}{2}$. Then the probabilities $\mathbb{P}(i_3 \to i_1) = O_{k}(\frac{1}{n})$ and $\mathbb{P}(i_3 \to i_2) = O_{k}(\frac{1}{n})$.
\end{lem}
\begin{proof}
In this proof, we will generally consider the computation of the probability of $i_3 \to i_1$. We use the same idea to show that the probability of the rewiring $i_3 \to i_2$ is also bounded by the same value. 
Let $0 \le l \le n$ be the number of all rewirings to vertex $i_1$ before $i_1$ is considered for the rewiring process. It follows the same computation with Lemma $6.4$. We have 
\begin{align*}
    \mathbb{P}(i_3 \xrightarrow{d} i_1) &= \sum_{t=0}^{n}\mathbb{P}(i_3 \xrightarrow{d} i_1 |\ l=t)\cdot \mathbb{P}(l=t) \\ &= \sum_{t=0}^{\frac{n}{2}}\mathbb{P}(i_3 \xrightarrow{d} i_1 |\ l=t)\cdot \mathbb{P}(l=t) + \sum_{t=\frac{n}{2}+1}^{n}\mathbb{P}(i_3 \xrightarrow{d} i_1 |\ l=t)\cdot \mathbb{P}(l=t) 
\end{align*}
By Lemma $6.7$, the probability $\mathbb{P}(i_3 \xrightarrow{d} i_1 |\ l=t) = O_{k}(\frac{1}{n})$, and it gives the term $\sum_{t=0}^{\frac{n}{2}}\mathbb{P}(i_3 \xrightarrow{d} i_1 |\ l=t)\cdot \mathbb{P}(l=t)$ is bounded by $O_{k}(\frac{1}{n})$. Since another term determines the case $l \ge \frac{n}{2}+1$, we have the term $\sum_{t=\frac{n}{2}+1}^{n}\mathbb{P}(i_3 \xrightarrow{d} i_1 |\ l=t)\cdot \mathbb{P}(l=t)$ has a bound $\mathbb{P}(l \ge \frac{n}{2}+1)$. Thus, by similar argument in Lemma $6.4$
\begin{align*}
    \mathbb{P}(i_3 \xrightarrow{d} i_1) &\le O_{k}(\frac{1}{n}) + \mathbb{P}(l \ge \frac{n}{2}+1) = O_{k}(\frac{1}{n}) + \sum_{t=\frac{n}{2}+1}^{n} \mathbb{P}(l=t) \\ &\le O_{k}(\frac{1}{n}) + \sum_{t=\frac{n}{2}+1}^{n} \binom{n}{t} \cdot \big(\frac{c}{n}\big)^{\frac{n}{4}} = O_{k}(\frac{1}{n}) + 2^{n} \cdot \big(\frac{c}{n}\big)^{\frac{n}{4}}, \mathrm{\ for\ some\ }c \\ &\le O_{k}(\frac{1}{n}) + O_{k}(\frac{1}{n^3}) \le O_{k}(\frac{1}{n}).
\end{align*}
For $1 \le d \le \frac{k}{2}$, it follows that
\begin{align*}
    \mathbb{P}(i_3 \to i_1) &= \sum_{d=1}^{\frac{k}{2}}\mathbb{P}(i_3 \xrightarrow{d} i_1) = (\frac{k}{2}) \cdot O_{k}(\frac{1}{n}) \le O_{k}(\frac{1}{n}).
\end{align*}
Since we have $||i_3-i_2|| > \frac{k}{2}$ and $i_2 < i_3$. The probability of the rewiring $i_3 \to i_2$ follows the same argument as $i_3 \to i_1$. Hence, the probability bound $\mathbb{P}(i_3 \to i_2) = O_{k}(\frac{1}{n})$.
\end{proof}
\begin{lem} For any distinct vertex $i_1,i_2,i_3$ in the $SW(n,k,p)$ random graph, those vertices satisfy the condition $||i_3-i_2|| > \frac{k}{2}$ and $||i_3-i_1|| > \frac{k}{2}$. Let $l \le \frac{n}{2}$ be the number of rewirings to vertex $i_3$ before this vertex is considered the rewiring process. Let $d \in \{1,2,...,\frac{k}{2}-1\}$ be fixed, and an edge $\{i_3,i_3+d\}$ is 
a downside edge of vertex $i_3$. For every $v = d+1,d+2,...,\frac{k}{2}$, then the following bound is true; $$\mathbb{P}(i_3 \xrightarrow{d} i_1, i_3 \xrightarrow{v} i_2 |\ l) = O_{k}(\frac{1}{n^2}).$$
\end{lem}
\begin{proof}
Let $d \in \{1,2,...,\frac{k}{2}-1\}$. The downside edge $\{i_3, i_3+d\}$ rewires from $i_3$ to vertex $i_1$. Let $v > d$ be the next downside edge $\{i_3,i_3+v \}$ of vertex $i_3$ that rewires from $i_3$ to vertex $i_2$. Let $c$ be the number of edges between $\{i_3, i_3+d\}$ and $\{i_3,i_3+v\}$. Hence, 
\begin{align*}
    &\mathbb{P}(i_3 \xrightarrow{d} i_1, i_3 \xrightarrow{d+1} i_2\ |\ l) \\ &= \mathbb{P}(\mathrm{all \ } d-1 \mathrm{\ edges \ do\ not \ rewire}) +  \mathbb{P}(\mathrm{only \ one \ edge \ rewires \ to \ not \ }i_1,i_2) \\ &+ \mathbb{P}(\mathrm{two \ edges \ rewire \ to \ not \ }i_1,i_2) + \mathbb{P}(\mathrm{three \ edges \ rewire \ to \ not \ }i_1,i_2) \\ &+...+ \mathbb{P}(\mathrm{all \ } d-1 \mathrm{\ edges \ rewire \ to \ not \ }i_1,i_2) \\ &= (1-p)^{d-1}\left(\frac{p}{n-k-2-l}\right)^2 \\ &+ \binom{d-1}{1}(1-p)^{d-2}\cdot\frac{(n-k-3-l)p}{n-k-1-l}\left(\frac{p}{n-k-3-l}\right)^2 \\ &+
    \binom{d-1}{2}(1-p)^{d-3}\cdot\frac{(n-k-3-l)p}{n-k-1-l}\cdot\frac{(n-k-4-l)p}{n-k-2-l}\left(\frac{p}{n-k-4-l}\right)^2 \\ &+ \binom{d-1}{3}(1-p)^{d-4}\cdot\frac{(n-k-3-l)p}{n-k-1-l}\cdot\frac{(n-k-4-l)p}{n-k-2-l}\cdot\frac{(n-k-5-l)p}{n-k-3-l} \\ &\cdot \left(\frac{p}{n-k-5-l}\right)^2 +...+ \\ &+ \binom{d-1}{d-1}\frac{(n-k-3-l)p}{n-k-1-l}\cdot \cdot \cdot \frac{[n-k-(d+1)-l]p}{[n-k-(d-1)-l]}\left(\frac{p}{[n-k-(d+1)-l]}\right)^2 \\ &= (1-p)^{d-1}\left(\frac{p}{n-k-2-l}\right)^2 \\ &+ \binom{d-1}{1}(1-p)^{d-2}\cdot\frac{p^3}{(n-k-1-l)(n-k-3-l)} \\ &+ \sum_{i=2}^{d-1} \binom{d-1}{i}(1-p)^{d-1-i}\frac{[n-k-(i+1)-l]p^{2+i}}{(n-k-1-l)(n-k-2-l)[n-k-(2+i)-l]} \\ &\le \left(\frac{1}{n-k-2-l}\right)^2 + (d-1)\left(\frac{1}{n-k-2-l}\right)^2 \\ &+ \left(\frac{1}{n-k-2-l}\right)^2 \left(\sum_{i=2}^{d-1} \binom{d-1}{i}\right) \\ &= O_{k}(\frac{1}{n^2}).
\end{align*}
Then we show the case for any $d < v \le \frac{k}{2}$ and $0 \le c = v-d-1$. It is divided into sub-cases of the number of edges from $d-1$ edges $\{i_3,i_3+1\},...,\{i_3,i_3+(d-1)\}$. Let $m \in \{0,1,2,...,d-1\}$ be the number of the previous $d-1$ edges that already rewire. In the computation, we fix the number $m$ and consider every rewiring condition of other $c$ edges. Let the notation $\mathbb{P}(m,c\ |\ l)$ be the probability of rewiring $i_3 \xrightarrow{d} i_1, i_3 \xrightarrow{v} i_2$ which some $m \le d-1$ downside neighbors of vertex $i_3$ already rewired and some $c$ downside neighbors of vertex $i_3$ also rewired. This probability is conditioning on $l$ upside vertices rewired to vertex $i_3$. It follows that 
\begin{align*}
    \mathbb{P}(0,c\ |\ l) &= \binom{d-1}{0}(1-p)^{d-1}\cdot\frac{p}{n-k-2-l}\Bigg[\binom{c}{0}(1-p)^c\cdot\frac{p}{n-k-2-l} \\ &+ \binom{c}{1}(1-p)^{c-1}\cdot\frac{(n-k-3-l)p}{n-k-2-l}\cdot\frac{p}{n-k-3-l} +...+ \\ &+ \binom{c}{c-1}(1-p)\cdot\frac{(n-k-3-l)p}{n-k-2-l}\cdot \cdot \cdot\frac{[n-k-(c+1)-l]p}{n-k-c-l}\\ &\cdot\frac{p}{n-k-(c+1)-l} + \binom{c}{c}\frac{(n-k-3-l)p}{n-k-2-l}\cdot\frac{(n-k-4-l)p}{n-k-3-l}\cdot \cdot \cdot \\ &\cdot \frac{[n-k-(c+2)-l]p}{n-k-(c+1)-l}\cdot\frac{p}{n-k-(c+2)-l} \Bigg]\\ &= \binom{d-1}{0}(1-p)^{d-1}\cdot\frac{p}{n-k-2-l}\left(\frac{p}{n-k-2-l}\right)\cdot \\ &\cdot \left(\sum_{i=1}^{c}\binom{c}{i}(1-p)^{c-i}p^i\right).
\end{align*}
By the binomial theorem, $$\sum_{i=1}^{c}\binom{c}{i}(1-p)^{c-i}\cdot p^i = [(1-p)+p]^c = 1^c = 1.$$ 
Hence, we have
\begin{align*}
    \mathbb{P}(0,c\ |\ l) &= \binom{d-1}{0}(1-p)^{d-1}\left(\frac{p}{n-k-2-l}\right)^2 \le O_{k}(\frac{1}{n^2}).
\end{align*}
Similarly, 
\begin{align*}
    \mathbb{P}(1,c\ |\ l) &= \binom{d-1}{1}(1-p)^{d-2}\cdot\frac{(n-k-3-l)p}{n-k-1-l}\cdot\frac{p}{n-k-3-l} \cdot \\ &\Bigg[\binom{c}{0}(1-p)^c\cdot\frac{p}{n-k-3-l} + \binom{c}{1}(1-p)^{c-1}\cdot\frac{(n-k-4-l)p}{n-k-3-l}\cdot \\ &\cdot \frac{p}{n-k-4-l} +...+ \binom{c}{c-1}(1-p)\cdot\frac{(n-k-4-l)p}{n-k-3-l}\cdot \cdot \cdot \\ &\cdot \frac{[n-k-(c+2)-l]p}{n-k-(c+1)-l}\cdot\frac{p}{n-k-(c+2)-l} \\ &+ \binom{c}{c}\frac{(n-k-4-l)p}{n-k-3-l}\cdot\frac{(n-k-5-l)p}{n-k-4-l}\cdot \cdot \cdot \frac{[n-k-(c+3)-l]p}{n-k-(c+2)-l} \cdot \\ &\cdot \frac{p}{n-k-(c+3)-l} \Bigg] \\ &= \binom{d-1}{1}(1-p)^{d-2}\cdot\frac{p^2}{n-k-1-l}\left(\frac{p}{n-k-3-l}\right)\left(\sum_{i=1}^{c}\binom{c}{i}(1-p)^{c-i}p^i\right) \\ &= \binom{d-1}{0}(1-p)^{d-1}\left(\frac{p^3}{(n-k-1-l)(n-k-3-l)}\right) \le O_{k}(\frac{1}{n^2}).
\end{align*}
For $d-1 \ge m \ge 2$, it follows that
\begin{align*}
    &\mathbb{P}(m,c\ |\ l) \\ &= \binom{d-1}{m}(1-p)^{d-1-m}
    \frac{(n-k-3-l)p}{n-k-1-l}\frac{(n-k-4-l)p}{n-k-2-l}\cdot \cdot \cdot \\ &\cdot\frac{[n-k-(m+2)-l]p}{n-k-m-l} \cdot \frac{p}{n-k-(m+2)-l} \Bigg[\binom{c}{0}(1-p)^c\frac{p}{n-k-(m+2)-l} \\ &+ \binom{c}{1}(1-p)^{c-1}\frac{[n-k-(m+3)-l]p}{n-k-(m+2)-l}\cdot\frac{p}{n-k-(m+3)-l} +...+ \\ &+ \binom{c}{c-1}(1-p)\cdot\frac{[n-k-(m+3)-l]p}{n-k-(m+2)-l}\cdot \cdot \cdot\frac{[n-k-(m+c+2)-l]p}{n-k-(m+c+1)-l}\cdot \\ &\cdot\frac{p}{n-k-(m+c+2)-l} + \binom{c}{c}\frac{[n-k-(m+3)-l]p}{n-k-(m+2)-l}\cdot \cdot \cdot \\ &\cdot \frac{[n-k-(m+c+3)-l]p}{n-k-(m+c+2)-l}\cdot\frac{p}{n-k-(m+c+3)-l} \Bigg] \\ &= \binom{d-1}{m}(1-p)^{d-1-m}\cdot\frac{[n-k-(m+1)-l]p^{m+1}}{(n-k-1-l)(n-k-2-l)} \cdot \\ &\left[\sum_{i=1}^{c}\binom{c}{i}(1-p)^{c-i}p^i\left(\frac{p}{n-k-(m+2)-l}\right)\right] \\ &= \binom{d-1}{m}(1-p)^{d-1-m}\cdot\frac{[n-k-(m+1)-l]p^{m+2}}{(n-k-1-l)(n-k-2-l)[n-k-(m+2)-l]} \\ &\le O_{k}(\frac{1}{n^2}).
\end{align*}
Therefore, we have the probability bound
$$\mathbb{P}(i_3 \xrightarrow{d} i_1, i_3 \xrightarrow{v} i_2 |\ l) =  \mathbb{P}(0,c |\ l) +  \mathbb{P}(1,c |\ l) + \sum_{m=2}^{d-1}\mathbb{P}(m,c |\ l) = O_{k}(\frac{1}{n^2}).$$
\end{proof}
\begin{lem}
    Let $i_1,i_2,i_3$ be the distinct vertices in the $SW(n,k,p)$ random graph. Those vertices must satisfy the condition $||i_3-i_1|| > \frac{k}{2}$ and $||i_3-i_2|| > \frac{k}{2}$. Then the probability $$\mathbb{P}(i_3 \to i_1, i_3 \to i_2) = O_{k}(\frac{1}{n^2}).$$
\end{lem}
\begin{proof} This proof follows almost immediately from Lemma $6.9$. Let $0 \le l \le n$ be the number of all rewirings to vertex $i_3$ before this vertex is considered for the rewiring process. It follows the same computation with Lemmas $6.4$ and $6.8$ when we compute the conditional probability given $l$. We have
\begin{align*}
    \mathbb{P}(i_3 \xrightarrow{d} i_1, i_3 \xrightarrow{v} i_2) &= \sum_{t=0}^{n}\mathbb{P}(i_3 \xrightarrow{d} i_1, i_3 \xrightarrow{v} i_2 |\ l=t) \cdot \mathbb{P}(l=t) \\ &= \sum_{t=0}^{\frac{n}{2}}\mathbb{P}(i_3 \xrightarrow{d} i_1, i_3 \xrightarrow{v} i_2 |\ l=t) \cdot \mathbb{P}(l=t) \\ &+ \sum_{t=\frac{n}{2}+1}^{n}\mathbb{P}(i_3 \xrightarrow{d} i_1, i_3 \xrightarrow{v} i_2 |\ l=t) \cdot \mathbb{P}(l=t) \\ &\le O_{k}(\frac{1}{n^2}) + \mathbb{P}(l \ge \frac{n}{2}+1),
\end{align*}
This inequality holds because Lemma $6.9$ tells us that $\mathbb{P}(i_3 \xrightarrow{d} i_1, i_3 \xrightarrow{v} i_2 |\ l=t) = O_{k}(\frac{1}{n^2})$. It makes the term $\sum_{t=0}^{\frac{n}{2}}\mathbb{P}(i_3 \xrightarrow{d} i_1, i_3 \xrightarrow{v} i_2 |\ l=t) \cdot \mathbb{P}(l=t)$ is bounded by $O_{k}(\frac{1}{n^2})$. For $l \ge \frac{n}{2}+1$, the term $\sum_{t=\frac{n}{2}+1}^{n}\mathbb{P}(i_3 \xrightarrow{d} i_1, i_3 \xrightarrow{v} i_2 |\ l=t) \cdot \mathbb{P}(l=t)$ is bounded by $\mathbb{P}(l \ge \frac{n}{2}+1)$. Hence, by  similar argument in Lemma $6.4$
\begin{align*}
    \mathbb{P}(i_3 \xrightarrow{d} i_1, i_3 \xrightarrow{v} i_2) &\le O_{k}(\frac{1}{n^2}) + \sum_{t=\frac{n}{2}+1}^{n} \binom{n}{t} \cdot \big(\frac{c}{n}\big)^{\frac{n}{4}}, \mathrm{\ for\ some\ }c \\ &= O_{k}(\frac{1}{n^2}) + 2^n \cdot \big(\frac{c}{n}\big)^{\frac{n}{4}} \\ &\le O_{k}(\frac{1}{n^2}) + O_{k}(\frac{1}{n^3}) \le O_{k}(\frac{1}{n^2}).
\end{align*}
For $d \in \{1,2,...,\frac{k}{2}\}$ and $v = d+1,d+2,...,\frac{k}{2}$, it follows that
\begin{align*}
    \mathbb{P}(i_3 \to i_1, i_3 \to i_2) &= \sum_{d=1}^{\frac{k}{2}-1}\sum_{v=d+1}^{\frac{k}{2}}\left[\mathbb{P}(i_3 \xrightarrow{d} i_1, i_3 \xrightarrow{v} i_2) + \mathbb{P}(i_3 \xrightarrow{d} i_2, i_3 \xrightarrow{v} i_1) \right]\\ &\le \sum_{d=1}^{\frac{k}{2}-1}\sum_{v=d+1}^{\frac{k}{2}} 2.\max\{\mathbb{P}(i_3 \xrightarrow{d} i_1, i_3 \xrightarrow{v} i_2), \mathbb{P}(i_3 \xrightarrow{d} i_2, i_3 \xrightarrow{v} i_1)\}\\ &\le \sum_{d=1}^{\frac{k}{2}-1}\sum_{v=d+1}^{\frac{k}{2}} O_{k}(\frac{1}{n^2}) \\ &\le O_{k}(\frac{1}{n^2}).
\end{align*}
\end{proof}
\begin{lem} Given the case \textbf{two far} configuration. Let $1 \le i_1 < i_2 < i_3 \le n$ be distinct vertices in the $SW(n,k,p)$ random graph. The vertices satisfy one of the following three conditions;
\begin{enumerate}
    \item  $||i_1-i_2|| \le \frac{k}{2},||i_2-i_3|| > \frac{k}{2},||i_3-i_1|| > \frac{k}{2}$
    \item  $||i_1-i_2|| > \frac{k}{2},||i_2-i_3|| \le \frac{k}{2},||i_3-i_1|| > \frac{k}{2}$
    \item  $||i_1-i_2|| > \frac{k}{2},||i_2-i_3|| > \frac{k}{2},||i_3-i_1|| \le \frac{k}{2}$
\end{enumerate}
Then the probability $\mathbf{P}_3 = O_{k}(\frac{1}{n^2})$.
\end{lem}
\begin{proof}
Let $\mathbf{P}_{3,i}$ be a maximum probability of the above condition $i$ for \textbf{two far} configuration. We will prove the probability bound  $\mathbf{P}_{3,1}$ and then use the result to come up with others probabilities $\mathbf{P}_{3,2}$ and $\mathbf{P}_{3,3}$. First, we choose the vertices $i_1,i_2,i_3$ that satisfy $||i_1-i_2|| \le \frac{k}{2},||i_2-i_3|| > \frac{k}{2},||i_3-i_1|| > \frac{k}{2}$. We assume that those vertices give a maximum probability of the first condition for \textbf{two far}configuration.
To have three connected vertices, an edges $\{i_1,i_2\}$ does not rewire and two other edges  $\{i_2,i_3\}, \{i_3,i_1\}$ are rewired. The construction of two rewiring edges can occur with four possibilities. \\ (1.) Rewiring $i_1 \to i_3$ and $i_2 \to i_3$. (2.) Rewiring $i_1 \to i_3$ and $i_3 \to i_2$. \\ (3.) Rewiring $i_3 \to i_1$ and $i_2 \to i_3$. (4.) Rewiring $i_3 \to i_1$ and $i_3 \to i_2$. \\ We have the following
\begin{align*}
    \mathbf{P}_{3,1} &= \mathbb{P}(A_{i_1i_2}=1,A_{i_2i_3}=1,A_{i_3i_1}=1) \\ &= \mathbb{P}(A_{i_1i_2}=1, i_1 \to i_3, i_2 \to i_3) + \mathbb{P}(A_{i_1i_2}=1, i_1 \to i_3, i_3 \to i_2) \\ &+ \mathbb{P}(A_{i_1i_2}=1, i_3 \to i_1, i_2 \to i_3) + \mathbb{P}(A_{i_1i_2}=1, i_3 \to i_1, i_3 \to i_2) \\ &= \mathbb{P}(A_{i_1,i_2}=1)\cdot \mathbb{P}(i_1 \to i_3 | A_{i_1,i_2}=1)\cdot \mathbb{P}(i_2 \to i_3) \\ &+ \mathbb{P}(A_{i_1,i_2}=1)\cdot \mathbb{P}(i_1 \to i_3 | A_{i_1,i_2}=1)\cdot \mathbb{P}(i_3 \to i_2) \\ &+ \mathbb{P}(A_{i_1,i_2}=1)\cdot \mathbb{P}(i_2 \to i_3)\cdot \mathbb{P}(i_3 \to i_1) \\ &+ \mathbb{P}(A_{i_1,i_2}=1)\cdot \mathbb{P}(i_3 \to i_1, i_3 \to i_2).
\end{align*}
By Lemma $6.4$, since we know that $||i_2-i_3|| > \frac{k}{2}$, we have the probability $\mathbb{P}(i_2 \to i_3) = O_{k}(\frac{1}{n})$. We also know $||i_1-i_2|| \le \frac{k}{2}$ and $||i_1-i_3|| > \frac{k}{2}$. Lemma $6.5$ tells us the probability $\mathbb{P}(i_1 \to i_3 | \ A_{i_1i_2}=1) = O_{k}(\frac{1}{n})$. In addition, we use the result from Lemma $6.8$  that the probability $\mathbb{P}(i_3 \to i_1) = O_{k}(\frac{1}{n})$ and $\mathbb{P}(i_3 \to i_2) = O_{k}(\frac{1}{n})$ because of $||i_3-i_1|| > \frac{k}{2}$ and $||i_3-i_2|| > \frac{k}{2}$. We plug in the results from above and Lemma $6.10$ to the formula of the bound of $\mathbf{P}_{3,1}$. Thus,
\begin{align*}
    \mathbf{P}_{3,1} &\le (1-p)O_{k}(\frac{1}{n})O_{k}(\frac{1}{n}) + (1-p)O_{k}(\frac{1}{n})O_{k}(\frac{1}{n}) + (1-p)O_{k}(\frac{1}{n})O_{k}(\frac{1}{n}) \\ &+ (1-p)O_{k}(\frac{1}{n^2}) \\ &\le O_{k}(\frac{1}{n^2}). 
\end{align*}
The other two conditions have the similar construction of three connected edges, which are two far edges and one close edge. We follows the same computation as the first condition. It gives us the same bound of the probabilities
\begin{align*}
\mathbf{P}_{3,2} &= \max\{\mathbb{P}(A_{i_1i_2}= 1, A_{i_2i_3} = 1, A_{i_3i_1} = 1)\ | \ i_1,i_2,i_3 \mathrm{\ satisfy\ second\ condition}\} \\ &= O_{k}(\frac{1}{n^2}),
\end{align*}
\begin{align*}
\mathbf{P}_{3,3} &= \max\{\mathbb{P}(A_{i_1i_2}= 1, A_{i_2i_3} = 1, A_{i_3i_1} = 1)\ | \ i_1,i_2,i_3 \mathrm{\ satisfy\ third\ condition}\} \\ &= O_{k}(\frac{1}{n^2}).
\end{align*}
Since all three probabilities have the same bound, the probability of the case \textbf{two far} configuration is $\mathbf{P}_{3} = \max\{\mathbf{P}_{3,1}, \mathbf{P}_{3,2}, \mathbf{P}_{3,3}\} = O_{k}(\frac{1}{n^2})$.
\end{proof}
\indent Then we compute the probability of the case \textbf{all far} configuration for distinct vertices $i_1,i_2,i_3$ in the $SW(n,k,p)$ random graph. We assume that three vertices are located in order $1 \le i_1 < i_2 < i_3 \le n$ and satisfy the condition $||i_1-i_2|| > \frac{k}{2},||i_2-i_3|| > \frac{k}{2},||i_3-i_1|| > \frac{k}{2}$.
\begin{lem} The probability $\mathbf{P}_4 = O_{k}(\frac{1}{n^3})$.
\end{lem}
\begin{proof} 
Let $i_1,i_2,i_3$ be distinct vertices in $SW(n,k,p)$ random graph. The vertices satisfy the above condition $||i_1-i_2|| > \frac{k}{2},||i_2-i_3|| > \frac{k}{2},||i_3-i_1|| > \frac{k}{2}$. We assume that the chosen vertices give a maximum probability of \textbf{all far} configuration.
There are eight possibilities of rewiring to have $\{i_1,i_2\}, \{i_2,i_3\}, \{i_3,i_1\}$ and all are far edges: \\
(1.) Rewiring $i_1 \to i_2, i_2 \to i_3, i_3 \to i_1$. 
(2.) Rewiring $i_2 \to i_1, i_3 \to i_2, i_1 \to i_3$. \\
(3.) Rewiring $i_1 \to i_2, i_2 \to i_3, i_1 \to i_3$.
(4.) Rewiring $i_1 \to i_2, i_3 \to i_2, i_3 \to i_1$. \\
(5.) Rewiring $i_1 \to i_2, i_3 \to i_2, i_1 \to i_3$.  
(6.) Rewiring $i_2 \to i_1, i_2 \to i_3, i_3 \to i_1$. \\
(7.) Rewiring $i_2 \to i_1, i_2 \to i_3, i_1 \to i_3$. 
(8.) Rewiring $i_2 \to i_1, i_3 \to i_2, i_3 \to i_1$. \\
\indent Consider the case each vertex rewires once to another vertex. Since the distance on the torus between each pair of two vertices are far $(\ge \frac{k}{2})$, by Lemma $6.8$, it follows that $\mathbb{P}(i_1 \to i_2) = O_{k}(\frac{1}{n})$, $\mathbb{P}(i_2 \to i_3) = O_{k}(\frac{1}{n})$, $\mathbb{P}(i_3 \to i_1) = O_{k}(\frac{1}{n})$, $\mathbb{P}(i_1 \to i_3) = O_{k}(\frac{1}{n})$, $\mathbb{P}(i_3 \to i_2) = O_{k}(\frac{1}{n})$, and $\mathbb{P}(i_2 \to i_1) = O_{k}(\frac{1}{n})$.
In addition, one vertex can rewire to both other two vertices. By Lemma $6.10$, we have $\mathbb{P}(i_1 \to i_2, i_1 \to i_3) = O_{k}(\frac{1}{n^2})$, $\mathbb{P}(i_2 \to i_1, i_2 \to i_3) = O_{k}(\frac{1}{n^2})$, and $\mathbb{P}(i_3 \to i_1, i_3 \to i_2) = O_{k}(\frac{1}{n^2})$. Thus, the bound of the probability is 
\begin{align*}
    \mathbf{P}_4 &= \mathbb{P}(i_1 \to i_2, i_2 \to i_3, i_3 \to i_1) + \mathbb{P}(i_2 \to i_1, i_3 \to i_2, i_1 \to i_3) \\ &+ \mathbb{P}(i_1 \to i_2, i_2 \to i_3, i_1 \to i_3) + \mathbb{P}(i_1 \to i_2, i_3 \to i_2, i_3 \to i_1) \\ &+ \mathbb{P}(i_1 \to i_2, i_3 \to i_2, i_1 \to i_3) + \mathbb{P}(i_2 \to i_1, i_2 \to i_3, i_3 \to i_1) \\ &+ \mathbb{P}(i_2 \to i_1, i_2 \to i_3, i_1 \to i_3) + \mathbb{P}(i_2 \to i_1, i_3 \to i_2, i_3 \to i_1) \\ &= \mathbb{P}(i_1 \to i_2)\cdot \mathbb{P}(i_2 \to i_3)\cdot \mathbb{P}(i_3 \to i_1) + \mathbb{P}(i_2 \to i_1)\cdot \mathbb{P}(i_3 \to i_2)\cdot\mathbb{P}(i_1 \to i_3) \\ &+ \mathbb{P}(i_1 \to i_2, i_1 \to i_3)\cdot \mathbb{P}(i_2 \to i_3) + \mathbb{P}(i_1 \to i_2)\cdot \mathbb{P}(i_3 \to i_2, i_3 \to i_1) \\ &+ \mathbb{P}(i_1 \to i_2, i_1 \to i_3)\cdot \mathbb{P}(i_3 \to i_2) + \mathbb{P}(i_2 \to i_1, i_2 \to i_3)\cdot \mathbb{P}(i_3 \to i_1) \\ &+ \mathbb{P}(i_2 \to i_1, i_2 \to i_3)\cdot \mathbb{P}(i_1 \to i_3) + \mathbb{P}(i_2 \to i_1)\cdot \mathbb{P}( i_3 \to i_2, i_3 \to i_1)
\end{align*}
We plug in the above results to the formula $\mathbf{P}_4$. Thus, we have
\begin{align*}
    \mathbf{P}_4 &\le 2 \cdot O_{k}(\frac{1}{n}) \cdot O_{k}(\frac{1}{n}) \cdot O_{k}(\frac{1}{n}) + 6 \cdot O_{k}(\frac{1}{n^2}) \cdot O_{k}(\frac{1}{n}) \\ &\le O_{k}(\frac{1}{n^3}).
\end{align*}
\end{proof}
\section{The Counting Sum}
This section computes the number of permutations of three vertices $i_1,i_2,i_3$ in the $SW(n,k,p)$ random graph. Since each configuration we use a bound of the probability, we can compute all permutations of vertices $i_1,i_2,i_3$ which satisfy each one of four configurations. We divide this section into four different configurations which are \textbf{all close}, \textbf{one far}, \textbf{two far}, and \textbf{all far}. \\
\noindent \textbf{1.} Computation of $\mathbf{C}_1$ with the case \textbf{all close} \\
\indent Let $\{1,2,...,n\}$ be a set of vertices in $SW(n,k,p)$ random graph. There are $n$ possible choices from the vertex set assigned to be vertex $i_1$. Each choice of $i_1$ can have $k$ possible choices of vertex $i_2$. It seems that each choice of $i_2$ we can also assign vertex $i_3$ with $k$ possible choices. However, we need to reconsider the choices of placing vertex $i_3$ that satisfy the distance on the torus between $i_1$ and $i_3$ for $||i_1-i_3|| \le \frac{k}{2}$. \\
\indent We start computing on the downside neighbors of vertex $i_1$. If $i_2 = i_1+1$, there are $\frac{k}{2}-1$ choices on the upside and $\frac{k}{2}-1$ on the downside to put vertex $i_3$. If $i_2 = i_1+2$, there are $\frac{k}{2}-1$ choices on the upside and $\frac{k}{2}-2$ on the downside to put vertex $i_3$. We continue this procedure of computing until $i_2 = i_1 + \frac{k}{2}$ the left far edge of $i_1$ neighbors. In this case, there are still $\frac{k}{2}-1$ choices on the upside but $0$ choice on the downside to put vertex $i_3$. It makes sense because if assigning vertex $i_3$ on the downside of $i_2$, the distance between $i_1,i_3$ will be larger than $\frac{k}{2}$. Then we sum all choices of $i_3$ for each vertex $i_2$ and multiply by $2$ because it can possibly occur on the upside neighbors of vertex $i_1$. Hence,
\begin{align*}
    \mathbf{C}_1 &= 2n. \Bigg[ \big(0+(\frac{k}{2}-1)\big)+ \big(1+(\frac{k}{2}-1)\big) + \big(2+(\frac{k}{2}-1)\big) + ... + \big((\frac{k}{2}-1) \\ &+(\frac{k}{2}-1)\big) \Bigg] \\ &= 2n.\Bigg[\bigg(0+1+2+...+(\frac{k}{2}-1)\bigg) + (\frac{k}{2})(\frac{k}{2}-1) \Bigg] \\ &= 2n.\Bigg[\left(\frac{(\frac{k}{2})(\frac{k}{2}-1)}{2}\right) + (\frac{k}{2})(\frac{k}{2}-1)\Bigg] \\ &= n\bigg(\frac{k}{2}-1\bigg)\bigg(\frac{k}{2} + k\bigg) = n\bigg(\frac{k-2}{2}\bigg)\bigg(\frac{3k}{2}\bigg) = \frac{3}{4}nk(k-2).
\end{align*}
\noindent \textbf{2.} Computation of $\mathbf{C}_2$ with the case \textbf{one far}
\\  \indent For a vertex set $\{1,2,...,n\}$ of the random graph, there are $n$ choices to assign a first vertex $i_1$. Each choice of $i_1$ can assign another vertex $i_2$ with $k$ possible choices of $i_1$ neighbors which are $\frac{k}{2}$ choices on the upside and $\frac{k}{2}$ choices on the downside. Since each side of $i_1$ neighbors is likely similar to figure out the location of $i_2,i_3$, we can start computing the downside of $i_1$ neighbors and then multiply by $2$. In this case, we must consider the distance on the torus between $i_1,i_2$ and $i_2,i_3$ are less than than or equal to $\frac{k}{2}$, but the distance on the torus between $i_3,i_1$ is larger than $\frac{k}{2}$. 
If $i_2 = i_1+1$, then there is just $1$ choice for placing vertex $i_3$ at  vertex $i_1+(\frac{k}{2}+1)$. If $i_2 = i_1 +2$, there are $2$ choices for placing vertex $i_3$ at either vertex $i_1+(\frac{k}{2}+1)$ or $i_1+(\frac{k}{2}+2)$. We continue this computation until $i_2 = i_1+\frac{k}{2}$, a far edge of $i_1$'s downside neighbors. There are $\frac{k}{2}$ choices on the downside of vertex $i_2$. We sum all possible choices of $i_3$ for each vertex $i_2$ and then multiply by $2$. Hence,
\begin{align*}
    \mathbf{C}_2 &= 2n.\Bigg[1+2+...+\frac{k}{2}\Bigg] = 2n.\bigg(\frac{(\frac{k}{2})(\frac{k}{2}+1)}{2}\bigg) = n\big(\frac{k}{2}\big)\big(\frac{k+2}{2}\big) = \frac{nk}{4}(k+2) \\ &\le O_{k}(n).
\end{align*}
\noindent \textbf{3.} Computation of $\mathbf{C}_3$ with the case \textbf{two far}
\\ \indent Let $\{1,2,...,n\}$ be a vertex set in the random graph, there are $n$ choices from all vertices to be vertex $i_1$. Each choice of $i_1$ we can assign another vertex $i_2$ with $k$ possible choices of $i_1$'s neighbors, which are $\frac{k}{2}$ on the upside and $\frac{k}{2}$ on the downside. We initially consider the downside of $i_1$'s neighbors and then multiply by $2$ because each side of the neighbors is symmetrical. For each location of $i_2$, we must include all possible choices of $i_3$ that make the distance between $i_2,i_3$ and $i_1,i_3$ are greater than $\frac{k}{2}$. If $i_2 = i_1+1$, then there are $(n-k-1)-1$ choices to be vertex $i_3$ (not $i_1$ itself, its $k$ neighbors and $i_2$'s neighbors). If $i_2 = i_1+2$, there are $(n-k-1)-2$ choices of $i_3$. We continue to sum all number of $i_3$ choices until $i_2 = i_1+\frac{k}{2}$. \\ There are $(n-k-1)-(\frac{k}{2})$ choices for $i_3$. Hence,
\begin{align*}
    \mathbf{C}_3 &= 2n.\Bigg[\bigg((n-k-1)-1\bigg) + \bigg((n-k-1)-2\bigg) + ... + \bigg((n-k-1)-(\frac{k}{2})\bigg) \Bigg] \\ &= 2n.\Bigg[\bigg(\frac{k}{2}\bigg)\bigg(n-k-1\bigg) - \bigg(1+2+...+ \frac{k}{2}\bigg) \Bigg] \\ &= 2n.\Bigg[\bigg(\frac{k}{2}\bigg)\bigg(n-k-1\bigg) - \bigg(\frac{(\frac{k}{2})(\frac{k}{2}+1)}{2}\bigg) \Bigg] \\ &= nk\Bigg[n-k-1 - \frac{k}{4}-\frac{1}{2}\Bigg] \\ &= nk(n-\frac{5k}{4}-\frac{3}{2}) \\ &\le O_{k}(n^2).
\end{align*}
\noindent \textbf{4.} Computation of
$\mathbf{C}_4$ with the case \textbf{all far} \\ \indent In this part, we will compute the upper bound of the number of permutation of the case \textbf{all far} configuration. It starts with all $n$ choices from a vertex set $\{1,2,...,n\}$ to assign vertex $i_1$. Each choice of $i_1$ has $n-k-1$ possible choices of $i_2$ to make a distance on the torus between vertices $i_1$ and $i_2$ being greater than $\frac{k}{2}$. In each choice of $i_2$, we can find at most $n-k-1-(k+1)$ (not $i_1$'s neighbors and $i_2$'s neighbors) choices to assign vertex $i_3$. Hence,
\begin{align*}
    \mathbf{C}_4 \le n(n-k-1)[n-k-1-(k+1)] = n(n-k-1)(n-2k-2) \le O_{k}(n^3).
\end{align*}
\section{The Limiting Third Moment}
We end up the main proof of the limiting moment of the eigenvalue distribution of $SW(n,k,p)$ random graph. 
\begin{proof} By Lemma $5.16$, a new generalized formula of the third moment of the eigenvalue distribution is
 \begin{align*}
     \mathbb{E}[\frac{1}{n}\mathrm{Tr}(A_n)^3] &= \frac{1}{n}\bigg[\mathbf{C}_1\mathbf{P}_1 + O(\mathbf{C}_2\mathbf{P}_2) + O(\mathbf{C}_3\mathbf{P}_3) + O(\mathbf{C}_4\mathbf{P}_4)\bigg].
 \end{align*}
 Next, we use Lemmas $6.1$, $6.6$, $6.11$, and $6.12$ to plug in all probabilities of each configuration $\mathbf{P}_i$ and the number of all permutations of vertices satisfying each configuration $\mathbf{C}_i$ from \textbf{Section 7}, for $i=1,2,3,4$. Hence, it follows that
 \begin{align*}
     \mathbb{E}[\frac{1}{n}\mathrm{Tr}(A_n)^3] &= \frac{1}{n}\Bigg[\frac{3}{4}nk(k-2)(1-p)^3 + O_{k}(n)\cdot O_{k}(\frac{1}{n}) + O_{k}(n^2) \cdot O_{k}(\frac{1}{n^2}) \\ &+ O_{k}(n^3) \cdot O_{k}(\frac{1}{n^3})\Bigg] \\ &= \frac{3}{4}k(k-2)(1-p)^3 + O_{k}(\frac{1}{n}).
 \end{align*}
 The last equality holds by the simplification. Then, we take the limit of the third moment of the eigenvalue distribution as $n \to \infty$,
 \begin{align*}
     \lim_{n \to \infty} \mathbb{E}[\frac{1}{n}\mathrm{Tr}(A_n)^3] &= \frac{3}{4}k(k-2)(1-p)^3.
 \end{align*}
 Therefore, we prove Theorem 3.11. 
 \end{proof}

\section{Acknowledgements}
I give a special thank to Dr.Sean O'Roukre for the advice on my undergraduate honors thesis which is a draft version of this paper, and also all anonymous readers for helpful comments.

\bibliographystyle{plain}

\begin{thebibliography}{10}
\bibitem{1}
Albert, R., Barabási, A. L.
\newblock Statistical mechanics of complex networks. 
\newblock Reviews of Modern Physics., 74 (1):47-97, 2002.

\bibitem{2}
Barabási, A. L., Derenyiab, I., Farkasa, I., Jeongc, H., Nedac, Z., Oltvaie, Z. N., Ravaszc, E., Schubertf, A., Vicseka, T.
\newblock Networks in life: scaling properties and eigenvalue spectra.
\newblock Physica, A 314:25-34, 2002.

\bibitem{3}
Barrat, A., Weigt, M.
\newblock On the properties of small-world network models.
\newblock European Physical Journal B., 13 (3): 547–560, 2000

\bibitem{4}
Durrett, R.
\newblock Random Graph Dynamics.
\newblock Cambridge University Press, 2007.

\bibitem{5}
Gao, X., Lin, Y., Zhang, Z.
\newblock Eigenvalues for the transition matrix of a small-world scale-free network: Explicit expressions and applications.
\newblock Phys. Rev., E 91,  062808, 2015.

\bibitem{6}
Mirchev, M. J.
\newblock On the Spectra of Scale-free and Small-World Networks.
\newblock NSTC, 2017.

\bibitem{7}
Nakkirt, P.
\newblock Limiting Moments of the Eigenvalue Distribution of the Watts-Strogatz Random Graph, 
Undergraduate Honors Thesis, 2019

\bibitem{8}
Newman, M.
\newblock The Structure and Function of Complex Networks.
\newblock SIAM Review., 45 (2): 167–256, 2003.

\bibitem{9}
Song, H. F., Wang, X.
J.
\newblock Simple  distance-dependent formulation of the Watts-Strogatz model for directed and undirected small-world networks.
\newblock Phys. Rev., E 90, 062801, 2014.

\bibitem{10}
Strogatz, S. H., Watts, D. J.
\newblock Collective dynamics of small-world networks.
\newblock Nature, 393 (6684):440-442, 1998.

\bibitem{11}
Tao, T.
\newblock Topics in Random Matrix Theory.
\newblock American Mathematical Society, 2012.

\bibitem{12}
Watts, D. J.
\newblock Small Worlds: The Dynamics of Networks Between Order and Randomness.
\newblock Princeton University Press, 1999.
\end{thebibliography}
\nocite{*}

\end{document}